\documentclass{amsart}

\usepackage[usenames,dvipsnames]{color}
\definecolor{darkgreen}{rgb}{0,0.5,0}
\usepackage{url}
\usepackage[
        colorlinks, citecolor=darkgreen,
        backref,
        pdfauthor={Jennifer S. Balakrishnan, Amnon Besser, J. Steffen M"uller}, 
        pdftitle={Computing integral points on
        hyperelliptic curves using quadratic Chabauty}
]{hyperref}

\usepackage{comment}
\usepackage{enumerate}

\usepackage[all]{xy}
\usepackage{algorithmic,algorithm}

\newcommand{\Q}{\mathbb{Q}}
\renewcommand{\O}{\mathcal{O}}

\newcommand{\Z}{\mathbb{Z}}
\newcommand{\F}{\mathbb{F}}
\newcommand{\Frob}{\textrm{Frob}}

\newcommand{\seta}{\mathcal{A}}
\newcommand{\setz}{\mathcal{Z}}

\newcommand{\cycp}{\overline{p}}

\newcommand{\PP}{\mathbb{P}}
\newcommand{\Zip}{\Z[1/p]}
\DeclareMathOperator{\Res}{Res}
\newcommand{\Qp}{\Q_p}

\newcommand{\Qpb}{\bar{\Q}_p}

\newcommand{\dr}{\textup{dR}}
\newcommand{\hdr}{H_{\dr}}
\newcommand{\DD}{U}

\newcommand{\calF}{\mathcal{F}}

\newcommand{\XX}{\mathcal{X}}
\newcommand{\UU}{\mathcal{U}}
\newcommand{\KK}{\mathcal{K}}
\newcommand{\TT}{\mathcal{P}}
\newcommand{\bom}{\bar{\omega}}

\newcommand{\ord}{\operatorname{ord}}
\newcommand{\Div}{\operatorname{Div}}
\newcommand{\Spec}{\operatorname{Spec}}

\newcommand{\disc}{\operatorname{disc}}
\renewcommand{\div}{\operatorname{div}}

\newtheorem{theorem}{Theorem}[section]

\newtheorem{proposition}[theorem]{Proposition}
 
\newtheorem{lemma}[theorem]{Lemma}
\newtheorem{corollary}[theorem]{Corollary}
\theoremstyle{definition}
\newtheorem{definition}[theorem]{Definition}
\newtheorem{problem}{Problem}
\newtheorem{algo}[theorem]{Algorithm}

\newtheorem{remark}[theorem]{Remark}
\newtheorem{example}[theorem]{Example}

\numberwithin{equation}{section}

\date{\today}

\begin{document}
 \title[Computing integral points on hyperelliptic curves]{Computing integral points on
 hyperelliptic curves using quadratic Chabauty}
\author{Jennifer S. Balakrishnan}
\address{Jennifer S. Balakrishnan, Mathematical Institute, University of Oxford, Woodstock Road, Oxford OX2 6GG, UK}

\author{Amnon Besser}
 \address{Amnon Besser, Department of Mathematics\\Ben-Gurion University of the Negev\\P.O.B. 653\\Be'er-Sheva 84105\\Israel}

\author{J. Steffen M\"{u}ller}
\address{J. Steffen M\"{u}ller, Institut f\"ur Mathematik, Carl von Ossietzky
Universit\"{a}t Oldenburg, 26111 Oldenburg, Germany}

\begin{abstract}
    We give a method for the computation of integral points on a hyperelliptic curve of
    odd degree over the rationals whose genus equals the  Mordell-Weil rank of its
    Jacobian.
    Our approach consists of a combination of the $p$-adic
    approximation techniques introduced in previous work with the
    Mordell-Weil sieve.
\end{abstract}
\maketitle


\section{Introduction} \label{sec:intro}

Let $f\in\Z[x]$ be a separable polynomial of
degree at least 3. The problem of determining all integer solutions $(x,y)$ to the equation $$y^2=f(x)$$ is classical in nature and has been studied from several different viewpoints, yielding a toolbox of approaches.  Techniques from Diophantine approximation, which make use of \emph{archimedean} analysis, include linear forms in elliptic logarithms, $S$-unit equations, and Baker's method.  We discuss each of these techniques and their practical limitations -- usually arising from astronomical height bounds or working with number fields of large degree -- in Section~\ref{sec:previous} below.  

In the realm of \emph{non-archimedean} analysis, perhaps the most celebrated technique in the study of integral points is Coleman's interpretation \cite{Col85a} of the method of Chabauty \cite{Chab41}, allowing one to determine the \emph{rational} points on a curve whose Jacobian has Mordell-Weil rank less than its genus.  This is accomplished by computing a $p$-adic line integral, among whose zeros lie the rational points of the curve. The hypothesis that the rank is less than genus is essential, and in practice, it is perhaps more difficult to check this hypothesis than to carry out the actual construction of the $p$-adic integral.

Over the last decade, Kim has initiated an exciting program aimed at removing this restricting on rank, allowing the study of rational points on hyperbolic curves through the use of nonabelian geometric objects generalizing the role of the Jacobian in the Chabauty-Coleman method \cite{Kim05, bdckw}. Through delicate calculations in $p$-adic Hodge theory, this nonabelian Chabauty method can produce iterated $p$-adic integrals playing the role of the abelian integrals arising from the Jacobian of the curve \cite{kim10,BKK11}.

In the spirit of Kim's program, we previously gave a method~\cite{BBM14} based on $p$-adic height pairings to $p$-adically approximate the set
of integral points in the case when the curve has Jacobian with Mordell-Weil rank \emph{equal to} its genus. To state this more precisely, let us introduce some notation.  Consider the genus $g$ hyperelliptic curve $X$ which has an affine model given by the equation $y^2 = f(x)$ with $\deg f = 2g+1$, and let $J$ denote its Jacobian. Fix an odd prime $p$ such that $X$ has good reduction at $p$.   Define locally analytic functions $f_i \colon X(\Qpb)\to \Qpb$ by
\begin{equation*}
  f_i(z) = \int_\infty^z  \frac{x^i\,dx}{2y}, \qquad 0 \leq i \leq g - 1
\end{equation*}
and extend them linearly to functionals 
\[
  f_i \colon J(\Q) \otimes \Q \to \Q_p.
\]
The method of Chabauty and Coleman can be interpreted as follows: If the Mordell-Weil rank of $J(\Q)$ is less than $g$, then one can construct a function given by a linear combination of the $f_i$ on $J(\Q_p)$ that vanishes on $J(\Q)$. 
By restricting this function to $X(\Q_p)$, one can approximate the points $X(\Q)$.

In~\cite{BBM14} we showed that under the hypotheses that $J$ has ordinary reduction at $p$, the Mordell-Weil rank of $J(\Q)$ is equal to $g$ and the $f_i$ are linearly independent functionals on $J(\Q) \otimes \Q$, one can construct a locally analytic function $\rho$ that takes on a prescribed, finite set of constants $T$ on integral points of $X$. By setting $\rho$ equal to each of these values, we find, among the set of $p$-adic points, the integral points on the working affine model of $X$. The function $\rho$ arises naturally from a local decomposition of the global $p$-adic height pairing, a natural quadratic form on the Jacobian; we refer to this method of finding integral points as \emph{quadratic Chabauty}.

In this work, we show how quadratic Chabauty can be used in practice to provably find
all integral points on $X$ by combining it with the Mordell-Weil sieve.
Since $\rho$ can be written as a convergent power series on every residue disk, we can
explicitly determine the finitely many  solutions over $\Z_p$ to the equations
\begin{equation}\label{eq:rhot}
  \rho(z) = t,
\end{equation}
up to some finite precision $p^N$, as $t$ runs through $T$.
For this computation, we need to relate the global $p$-adic height to a natural basis of quadratic forms on $J(\Q) \otimes \Q$, compute double Coleman integrals describing the local height contribution at $p$, and calculate the arithmetic intersections describing the local height contributions away from $p$.  We briefly discussed how to compute these quantities in~\cite{BBM14}; there we also presented two examples. In this paper, we give a full algorithm describing the necessary computations and provide an analysis of the $p$-adic precision which must be maintained throughout the computation. We furthermore detail how to combine quadratic Chabauty with the Mordell-Weil sieve to precisely find the set of integral points. 

We note that the  algorithms for computing integral points on elliptic curves over $\Q$ are rather
well-developed and perform particularly well in the case of rank 1 (see
Section~\ref{sec:previous}), so we have just a few remarks here, combining quadratic Chabauty with information about the
group of $\F_q$-rational points for primes $q$ of good reduction.  This method is treated
in Appendix~\ref{sec:elliptic}.

In this paper, we focus on a method for curves having genus at least~2. 
Here the existing approaches usually require either the rank $r$ of $J(\Q)$ to be strictly
smaller than $g$ or the availability of a set of Mordell-Weil generators, which currently
is only possible for $g \le 3$; see Section~\ref{sec:previous}.  Our approach here, combining quadratic Chabauty with the Mordell-Weil sieve, 
thus presents the first systematic and practical method to compute integral points for
curves $X$ as above with $r = g \ge 4$.

In broad terms, the idea is as follows: We first find the integral points of small height;
our goal is to show that this set already contains all integral points.
Then we apply quadratic Chabauty using several primes, producing an integer $M$ and a list of
residue classes in $J(\Q)/MJ(\Q)$.  We want to show that each of these residue classes cannot contain the
image of a rational point on $X$. We accomplish this by applying  the Mordell-Weil sieve.  This method,
pioneered by Scharaschkin  \cite{Scharaschkin:Thesis},  combines information modulo several primes $v$ of good reduction by finding
the image of $X(\F_v)$ inside $J(\F_v)/MJ(\F_v)$.

The idea of using $p$-adic techniques, notably Chabauty's method, to come up with
congruence conditions and hence suitable lists of residue classes for the Mordell-Weil
sieve is not new, see~\cite{Bruin-Elkies:trinomials} and, in particular~\cite{PSS:Twists}, which also uses
congruence conditions modulo different primes.
In the present work, however, we choose the primes $v$ in the Mordell-Weil sieve computation and the
primes used for quadratic Chabauty {\em at the same time}, making it possible 
to keep the number of residue classes and/or the number (and size) of the primes $v$
comparatively small.
This is necessary, because we want to consider examples of large genus (and, hence, large
rank), and one typically needs to find about $v$ discrete logarithms in $J(\F_v)$, which
becomes slow for large $v$ and $g$.
See Example~\ref{g4}, where we use this approach to find the integral points on a curve of
genus~4; here $v=317$ is the largest prime we had to consider.

The structure of this paper is as follows: In Section~\ref{sec:overviewQC}, we review the method of quadratic Chabauty. In Section~\ref{sec:computing} we give a more detailed description of the algorithms needed to compute
the quantities in quadratic Chabauty and to solve equation \eqref{eq:rhot}. In Section~\ref{sec:cliffhanger}, we discuss some practical considerations arising from the method.  We present the Mordell-Weil sieve and explain how to apply it in our situation in
Sections~\ref{sec:mw-sieve}-\ref{sec:applying_mws}. This allows us to describe a complete method for the computation of integral points when $X$
satisfies $r=g>1$ in Section~\ref{sec:algorithm}.
Finally, we mention a few examples in Section~\ref{sec:examples}, one of which has genus~4
and rank~4 and is not amenable to previous algorithms.

Most of our algorithms can be generalized to hyperelliptic curves defined over number
fields. 
This is work in progress.

\subsection{Other approaches} \label{sec:previous}
We discuss here how our method compares to other approaches in the literature. For completeness, we also give an overview of existing methods in the case of genus 1 as well.

The most efficient technique for the computation of integral points on elliptic curves relies on linear forms in
elliptic logarithms.
These can be used to compute an upper bound on the height of integral
points~\cite{David95, HKK}, or, equivalently, on the size of the coefficients of an
integral point in terms of a given set of Mordell-Weil generators.
This bound can then be reduced using the LLL-algorithm, see for
instance~\cite{Smart:S-integral, PZGH99, Stroeker-Tzanakis}; frequently, the resulting bound is small enough to find all integral
points simply by searching up to this bound.
The computer algebra systems {\tt Sage}~\cite{sage} and {\tt Magma}~\cite{magma} contain
an implementation of this algorithm.

Quadratic Chabauty can be combined with this approach, as it can be used to provide a
lower bound on the height of all integral points that have not been found yet.
However, here we are restricted to the rank~1 case, where usually a generator of the
Mordell-Weil group modulo torsion is easily computed and the upper bounds after
LLL-reduction are small enough so that only a few multiples of the generator have to be
computed to provably find all integral points.

For genus greater than 1, no analogue of the algorithm based on elliptic logarithms is known. 
We briefly summarize the existing algorithms. 
More detail can be found in the introduction to~\cite{BMSST08};  see also~\cite{Smart98}.

The method of Chabauty-Coleman~\cite{Col85a} is a $p$-adic method arising from Coleman's reinterpretation of Chabauty's
theorem~\cite{Chab41}; see~\cite{PMC} for an introduction to the method as well as some further variations. It can often be used to compute $X(\Q)$ in practice,
provided $g$ is greater than the rank of $J(\Q)$.
It can also be combined with the Mordell-Weil sieve  \cite{Bruin-Stoll:MWSieve}, which we
discuss in detail in Section~\ref{sec:mw-sieve}. On its own, the Mordell-Weil sieve is
most useful to show that a given curve has {\em no} rational points, which is never the
case for our curves.

Another class of approaches to the computation of $X(\Q)$ is based on coverings of $X$, see for
instance~\cite{Flynn:Coverings, Bruin-Stoll:Descent}.
These require computations of the class group and unit group of number fields of potentially large
degree. 
In practice, such methods often combine the use of coverings with \emph{elliptic
curve Chabauty} \cite{Bruin:EllipticChabauty,Flynn-Wetherell:Challenge}, which requires the computation of Mordell-Weil groups of elliptic curves over
number fields of large degree. Such computations can be prohibitively expensive in practice, but sometimes work rather well for reasonably small
genus, rank and coefficients.

There are also methods which can be used to provably find all {\em integral} points.
These are based on Baker's method,  which can be used to compute an upper bound on the
height of integral points on $X$, which provides a theoretical algorithm for the
computation of the integral points.
While Baker's original bounds were too large to be of much practical use, a number of
authors have produced substantial improvements; see \cite[$\S 1$]{BMSST08} for an overview.
Some algorithms based on Baker's method are given in~\cite{Smart98}.

In~\cite{BMSST08}, a refinement of the Mordell-Weil sieve is discussed, which can be
used to show that any rational points on $X$ which have not been found yet must have
extremely large height.
It is possible to obtain lower bounds around $10^{2000}$ in this way.
The authors of {\em loc. cit.} also produce explicit upper bounds on the height of
integral points.
These are often good enough to show that all integral points have been found when combined
with the lower bounds.
The main obstacle is that for the refinement of the Mordell-Weil sieve one needs an
explicit set of generators of the Mordell-Weil group $J(\Q)$.
In practice, this is only possible when $g\le 3$, see~\cite{Flynn-Smart, Stoll:H1,
Stoll:H2,Mueller-Stoll, Stoll:KummerG3}.

The method we present here in Section~\ref{sec:algorithm} is applicable whenever we can
prove that the rank is equal to the genus and we have explicit generators of a subgroup of
$J(\Q)/J(\Q)_{\mathrm{tors}}$ of finite index.

\section*{Acknowledgements}
We would like to thank Nils Bruin, Tzanko Matev and Michael Stoll for
helpful discussions regarding the Mordell-Weil sieve and Michael Stoll for many helpful
comments on a previous version of this paper.
The second author was supported by Israel Science Foundation grant No. 1517/13. The third author was supported by DFG grant KU~2359/2-1.  


\section{Overview of quadratic Chabauty}\label{sec:overviewQC}
In this section, we give an overview of the method of quadratic Chabauty \cite{BBM14} for hyperelliptic curves and describe the $p$-adic analytic functions that allow us to find integral points.  Let $f\in\Z[x]$ be a separable polynomial of
degree $2g+1 \ge 3$.  The equation \[y^2=f(x)\] defines a hyperelliptic curve $X$ of genus $g$ over $\Q$
as follows:
Let $$\UU=\Spec(\Z[x, y] / (y^2 - f(x))).$$ Then $X$ is the normalization of the
projective closure of the generic fiber of $\UU$; it contains a unique point $\infty$ at infinity, and this point
is $\Q$-rational. We denote by $w \colon X\to X$ the hyperelliptic involution on $X$ and by
$\iota$ the embedding of $X$ into its Jacobian $J$ which maps a point $P$ to the class of
$(P) - (\infty)$.  Fix a branch of the $p$-adic logarithm, which we denote as $\log_p$.

We assume that $p$ is a prime of good reduction for
$X$. For a finite place $v$ of $\Q$ and divisors $D_1$
and $D_2$ on $X(\Q_v)$ of degree zero 
with disjoint support, let $$h_v(D_1,D_2) \in \Q_p$$ denote the local $p$-adic height
pairing at $v$ defined by Coleman and Gross~\cite{Col-Gro89} with respect to a direct sum decomposition of $\hdr^1(X\times \Q_p)$:

\begin{equation}\label{eq:decompi}
  \hdr^1(X\times \Q_p) = W\oplus \hdr^0(X \times \Q_p, \Omega^1_{X\times \Q_p}).
\end{equation}
The local $p$-adic height pairing is symmetric if and only if $W$ is isotropic with respect to the cup product pairing.  When $p$ is a prime of good \emph{ordinary} reduction, we take $W$ to be the unit root subspace for the action of Frobenius.

For $v \ne p$, the symmetric bilinear pairing $h_v(D_1,D_2)$ can be defined in terms of arithmetic intersection theory on a
regular model of $X\times\Q_v$ over $\Z_v$, while for $v=p$ it is given by a Coleman integral
along $D_2$ of a certain differential of the third kind with residue divisor $D_1$.

If $D_1$ and $D_2$ are defined over $\Q$, then the sum
\begin{equation*}
  h(D_1,D_2) = \sum_v h_v(D_1,D_2)
\end{equation*}
of local height pairings over the finite places $v$ of $\Q$ respects linear equivalence.
It induces the global Coleman-Gross $p$-adic  height pairing
\begin{equation*}
  h \colon J(\Q) \times J(\Q) \to \Qp.
\end{equation*}
This is a symmetric bilinear pairing, which is conjectured to be
nondegenerate.

In~\cite{Bes12}, it was shown that the restriction that 
$D_1$ and $D_2$ must have disjoint support can be removed 
by extending the local height pairing relative to a choice of tangent vectors. 
This theme was further explored in \cite{BBM14}, where we used a certain consistent choice of tangent vectors to study local heights.
The resulting local height pairing at $p$ gives 
rise to a function $$\tau(z)= h_p((z)-(\infty),(z)-(\infty))$$ defined on $X(\Q_p)$;
we can then extend $\tau$ to a function on $X(\Qpb)$.
This function turns out to be a Coleman function in the sense of~\cite{Bes00} and,
by~\cite[Theorem~2.2]{BBM14}, a double Coleman integral.
Hence it can be written as a convergent $p$-adic power series on every residue disk.

We define locally analytic functions $f_i \colon X(\Qpb)\to \Qpb$ by
\begin{equation*}
  f_i(z) = \int_\infty^z \omega_i, \qquad 0 \leq i \leq g-1
\end{equation*}
where $\omega_i=\frac{x^idx}{2y}$.
These $f_i$ can be extended linearly to functionals 
\[
  f_i \colon J(\Qpb) \otimes \Q \to \Qpb.
\]
Because Coleman integrals are compatible with the action
of the absolute Galois group of $\Q_p$, the restrictions of these functionals to
$J(\Q)\otimes\Q$, also denoted by $f_i$, are $\Q_p$-valued.

From now on, we make the assumption that $f$ does not reduce to a square modulo $q$ for any
prime number $q$.
In~\cite{BBM14} we proved the following result as Theorem~3.1:
\begin{theorem}\label{cutting-function}
 Suppose that the Mordell-Weil rank of $J(\Q)$ is $g$ and that the $f_i$ are linearly
 independent as functionals on $J(\Q)\otimes \Q$.
  Then there exist constants $\alpha_{ij}\in \Q_p$ such that the locally analytic function
  \begin{equation}\label{rhofunction}
    \rho(z) = \tau(z) - \sum_{0 \le i\le j \le g-1} \alpha_{ij} f_i(z)f_j(z)
  \end{equation}
  takes values on $\UU(\Zip)$ in an effectively computable  finite set $T\subset \Q_p$.
\end{theorem}
\begin{remark}In the statement of~\cite[Theorem~3.1]{BBM14} we incorrectly required only that no
linear combination of the $f_i$ vanished on $\UU(\Zip)$, but what is actually needed is the stronger
independence assumption stated here, which was indeed used in the proof of the theorem.\end{remark}
\begin{remark} 
  Note that we can apply a variant of the method of Chabauty-Coleman
if a linear combination of the $f_i$ vanishes on $\UU(\Zip)$,
 so the ``quadratic Chabauty'' situation described in the theorem is really the novel case.\end{remark}

We give a sketch of the proof of Theorem~\ref{cutting-function}.  
Because the rank is $g$ and the $f_i$ are linearly independent, the set
\begin{equation*}\left\{\frac{1}{2}(f_i f_j + f_j f_i)\right \}_{0 \leq i
  \leq j \leq g-1}\end{equation*} forms a basis of the space of $\Q_p$-valued bilinear
  forms on $J(\Q) \otimes \Q$. 
The global $p$-adic height pairing $h$ is also a $\Q_p$-valued bilinear form on $J(\Q) \otimes \Q$, so 
there are constants $\alpha_{ij}\in \Q_p$ such that 
\begin{equation}\label{eq:h_sum}
  h((P)-(\infty),(P)-(\infty)) = \sum_{i\le j} \alpha_{ij} f_i(P)f_j(P)
\end{equation}
for all $P \in X(\Q)$.

The right hand side of~\eqref{eq:h_sum} extends to a locally analytic $\Qpb$-valued function on
$X(\Qpb)$.
It then only remains to show that 
\[
  \rho(z) := \tau(z) - \sum_{i\le j} \alpha_{ij} f_i(z)f_j(z) 
\]
can only take values on $\UU(\Zip)\subset X(\Qpb)$ in an effectively computable  finite
subset of $\Q_p$. 
For this, we analyze the possible values of the local height pairings
$h_q((P)-(\infty),(P)-(\infty))$ for $q \ne p$ and $P \in \UU(\Zip)$, and use that by construction
\begin{equation}\label{eq:rho_away}
  \rho(P) = -\sum_{q \ne p}h_q((P)-(\infty), (P)-(\infty))
\end{equation}
for $P \in X(\Q)$.

To ease notation, we will only deal with a single prime $q$ and omit it from the
notation whenever possible.
Following~\cite[\S3]{BBM14}, we let $\XX$ denote a strong desingularization 
of the Zariski closure $\XX'$ of $X \times \Q_q$ over $\Spec(\Z_q)$. 
By the latter, we mean the scheme defined in the weighted projective plane 
$\PP^2_{\Z_q}(1,\,g+1,\,1)$ by the equation
\[
Y^2 = F(X,Z),
\]
where $F(X, Z)$ is the degree $2g+2$-homogenization of $f$.
By abuse of notation, the section in $\XX(\Z_q)$ corresponding to a point $P \in X(\Q_q)$
will also be denoted by $P$.
Our assumptions on $f$ guarantee that $\XX'$ is normal with irreducible and reduced special fiber.
Hence there is a unique component  of
the special fiber of $\XX$ which dominates the special fiber
of $\XX'$, 
and any $P \in X(\Q_q)$  whose reduction modulo $q$ is nonsingular has
the property that the section $P$ intersects this component.
Let $(\;\,.\;\,)$ denote the rational-valued intersection multiplicity on $\XX$.
For a divisor $Z$ on $X\times \Q_q$ of degree zero, we let $\Phi(Z)$ denote a vertical
$\Q$-divisor on $\XX$ such that $Z +\Phi(Z)$ has intersection multiplicity~0 with all vertical
divisors on $\XX$; such a $\Q$-divisor always exists by~\cite[Theorem~III.3.6]{Lan88}.
Then, by the proof of~\cite[Proposition~3.3]{BBM14}, we have
\begin{equation}\label{E:hqdp}
  h_q((P)-(\infty), (P)-(\infty)) = -(D_{P}\,.\,D_P)\log_p(q),
\end{equation}
where 
\begin{equation}\label{E:dp}
  D_{P} =(P) -( \infty) +\Phi((P)-(\infty))\in \Div(\XX)\otimes \Q,
\end{equation}
with the abuse of notation introduced above.
Here the self-intersection is defined with respect to the choice of tangent vectors
mentioned above; see the proof of~\cite[Proposition~3.3]{BBM14} and~\cite[\S5]{Gross86} for further details.

Let $T(q)\subset \Q$ denote the set of all self-intersections $D_P^2 = (D_P\,.\,D_P)$, where $P$ runs
through $\UU(\Z_q)$.
\begin{lemma}\label{L:tdec}
  \begin{enumerate}[(i)]
    \item The intersection multiplicity $D_P^2$ only depends on the fiber of $\XX_q$ which
      the section $P$ intersects.
    \item If  $T(q) \ne \{0\}$, then the special fiber
      $\XX_q$ is reducible or $q$ divides the leading coefficient of $f$. In particular,
      $T(q) =\{0\}$ if $q$ is a prime of good reduction.
\item If $q_1,\ldots,q_\ell$ are the primes $q$ for which $T(q) \ne \{0\}$, then
  \[
    \rho(\UU(\Zip)) \subset \left\{\sum^\ell_{i=1} d_i\log_p(q_i)\,:\, d_i \in T(q_i)\right\}.
  \]
  \end{enumerate}
\end{lemma}
\begin{proof}
  For the first two statements see the proof of~\cite[Proposition~3.3]{BBM14}.
  The proof of (iii) follows from~\eqref{eq:rho_away} and~\eqref{E:hqdp}.
\end{proof}
In particular, the value of $\rho$ at a $p$-integral point $P$ on $X$ is completely
determined by the intersection multiplicities $D_{P,q}^2\in \Q$ for primes $q$ of bad reduction and these
rational numbers do not depend on $p$!
We encode this global data as follows:
\begin{definition}\label{D:int_pat}
  We call a prime $q$ {\em very bad for $X$} if $T(q) \ne \{0\}$.
Suppose that $q_1< q_2<\ldots< q_\ell$ are the very bad primes for $X$.
We call a tuple
\[
  (d_1,\ldots,d_\ell) \in \Q^\ell
\]
an {\em intersection pattern} if $d_i \in T(q_i)$ for all $1 \le i \le \ell$.
Let $\TT$ denote the set of intersection patterns; this set is finite by
Lemma~\ref{L:tdec}.
We say a point $P \in \UU(\Z[1/p])$ {\em has intersection pattern $d =
(d_1,\ldots,d_\ell)\in\TT$}, if $D_{P,q_i}^2 = d_i$ for $i \in \{1,\ldots,\ell\}$.
Moreover, we define a $p$-adic logarithm map $\log_p \colon \TT \to \Q_p$ by
\[
  \log_p(d) = \sum^n_{i=1}d_i\log_p(q_i).
\]
If there are no very bad primes, then we set $\TT=\{1\}$ and accordingly $\log_p(\TT) =
\{0\}$.
\end{definition}
With this notation, the following result follows immediately from Lemma~\ref{L:tdec}:
\begin{corollary}\label{C:int_pat}
Let $T := \log_p(\TT) \subset \Q_p$. Then we have
  \[
    \rho(\UU(\Zip)) \subset T.
  \]
\end{corollary}


\section{Algorithms for quadratic Chabauty} \label{sec:computing}
Building on the theory described in the previous section, in this section we discuss how
to use Theorem~\ref{cutting-function} to $p$-adically approximate integral points in
practice. 
In $\S\S\ref{sec:alphas}-\ref{sec:T}$ we give algorithms to compute the set $T$ and  each of the quantities in the
equation giving the function $\rho(z)$:
\begin{equation}\label{eq:rho_eqn}
    \rho(z) := \underbrace{\tau(z)}_{\S\ref{sec:tau}}- \sum_{0 \le i\le j \le g-1}
    \underbrace{\alpha_{ij}}_{\S\ref{sec:alphas}}
    \underbrace{f_i(z)f_j(z)}_{\S\ref{sec:fi}} = \underbrace{t \in T_{}}_{\S\ref{sec:T}}.
  \end{equation}
 Finally, in $\S\ref{sec:solve}$, we discuss how to compute the solutions in $\UU(\Z_p)$
 to~\eqref{eq:rho_eqn}.
 By Theorem~\ref{cutting-function}, the set of solutions is finite and will contain the
 integral points on $X$. We will henceforth assume that $p$ is a prime of good ordinary reduction.

\subsection{Computing the $\alpha_{ij}$}\label{sec:alphas}The first step is to express the
global $p$-adic height function as a $\Q_p$-linear combination of a natural basis of the
space of $\Q_p$-valued quadratic forms on $J(\Q) \otimes \Q$.  To do this, we need 
independent non-torsion points $P_1,\ldots, P_g \in J(\Q)$.
In practice we usually have such points available if we know that $r=g$,
because the verification of this fact requires bounding $r$ from below, typically by
finding $g$ independent non-torsion points in $J(\Q)$.

We compute the global $p$-adic height pairings
$h(P_1,P_1),\, h(P_1, P_2), \ldots, h(P_g, P_g)$ using \cite[Algorithm 3.8]{BaMuSt12}; each $p$-adic height is a sum of Coleman-Gross local height integrals (a precision analysis of which is done in \cite[$\S 6$]{Bes-Bal10}), as well as rational multiples of $\log_p(q_i)$ for rational primes $q_i$ (which can be computed to the necessary precision using standard methods).  

Next, we check if the set 
$\{f_0,\ldots,f_{g-1}\}$ gives a basis of the $\Q_p$-valued functionals on $J(\Q)\otimes\Q$, or equivalently, if the set
\begin{equation}\label{qfbasis}\left\{\frac{1}{2}(f_i f_j + f_j f_i)\right \}_{0 \leq i
  \leq j \leq g-1}\end{equation} forms a basis for the space of $\Q_p$-valued bilinear
  forms on $J(\Q) \otimes \Q$. 
  We use the fact that we have, for each pair $(k,l)$, the
  relationship $$h(P_k, P_l) = \sum_{0 \leq i \leq j \leq g-1} \alpha_{ij} \cdot
  \left(\frac{1}{2}\left(f_i(P_k)f_j(P_l) + f_j(P_k)f_i(P_l)\right)\right),$$ if \eqref{qfbasis} gives a basis. In particular, by computing the determinant of the matrix whose entries are given by   $\left(\frac{1}{2}\left(f_i(P_k)f_j(P_l) + f_j(P_k)f_i(P_l)\right)\right)$ and verifying that it is nonzero, we determine that the $f_i$ are linearly independent $\Q_p$-valued functionals on $J(\Q) \otimes \Q$. If the $f_i$ aren't linearly independent, then we instead use the classical Chabauty-Coleman method and compute $X(\Q)$.
  
To set up the linear system relating the global height pairing $h$ to the elements in \eqref{qfbasis}, we use the computed values of $h(P_k, P_l)$ as well as the values of the sums of products of Coleman integrals  \begin{equation}\label{sumsofprods}\frac{1}{2}(f_i(P_k)f_j(P_l) + f_j(P_k)f_i(P_l)),\end{equation} computed as in \cite[Algorithm 11]{BBK09} to recover the $\alpha_{ij}$.  

We postpone a discussion of all precision concerning Coleman integrals to $\S\ref{sec:tau}$. The precision of the integrals present in  \eqref{sumsofprods} is discussed in Proposition \ref{prop19}.

\subsection{Computing $f_i(z)$}\label{sec:fi} The Coleman integral $$f_i(z) = \int_{\infty}^z \omega_i$$ of a regular 1-form $\omega_i$ is computed as a power series expansion for each residue disk. We compute the set of residue disks, then for each residue disk, do the following:
\begin{enumerate}
 \item Take a lift of the residue disk to a point $P \in \UU(\Z_p)$. In particular, if working in a Weierstrass residue disk, take the lift to be the characteristic zero Weierstrass point; else, take an arbitrary lift.
 \item Compute $f_i(P) = \int_{\infty}^P \omega_i =\frac{1}{2} \int_{-P}^P \omega_i$ using \cite[Lemma 16]{BBK09}, which sets the global constant of integration and reduces the computation of $f_i(z)$ to that of a tiny Coleman integral. 
\item Take $z$ to be a local coordinate in the disk of $P$ and break up the path between $\infty$ and $z$ using \begin{align*}f_i(z) &= \int_{\infty}^P \omega_i + \int_P^z \omega_i \\
 &= f_i(P) + \int_P^z \omega_i,\end{align*}
which allows $\int_P^z \omega_i$ to be computed as a tiny Coleman integral, as in \cite[Algorithm 8]{BBK09}.
 \end{enumerate}
 
Computing $f_i(P)$ requires the precision given in Proposition \ref{prop19}.

\subsection{Computing $\tau$}\label{sec:tau}
We begin with the observation that the local component of the $p$-adic height at $p$, $$\tau(z) = h_p((z) - (\infty), (z)-(\infty)),$$ can
be expressed as a sum of double Coleman integrals \cite[Theorem~2.2]{BBM14}:

$$\tau(z) = -2\int_{b_0}^z\sum_{i=0}^{g-1} \omega_i \bom_i,$$
where $b_0$ is a tangential basepoint at infinity, as in \cite[$\S 2$]{BBM14}.  Then computing the $p$-adic power series $\tau(z)$ amounts to \begin{itemize}\item
    computing the value $\tau(P)$ for an arbitrary point $P \in \UU(\Z_p)$ to set the
  global constant of integration and \item fixing a lift of Frobenius $\phi$, computing the the action of $\phi$ on the first Monsky-Washnitzer cohomology of an appropriate affine piece of the curve, which yields the
    analytic continuation of the double integral to each of the other residue disks not containing $P$. \end{itemize}

In practice, what is computationally easier is to compute the set of residue disks, then
for each residue disk, do the following:

\begin{enumerate}
\item Take a lift to a point $P \in \UU(\Z_p)$ (if $P$ is in a Weierstrass disk, take the lift to be the characteristic zero Weierstrass point; else, take an arbitrary lift) and compute $\tau(P)$, setting the constant of integration.
\item Compute the dual basis $\{\bom_i\} \subset W$ for $\omega_0,\ldots,\omega_{g-1}$
    with respect to the cup product pairing. This is done by \begin{enumerate}\item first
        computing, for $j = g, \ldots, 2g-1$, the differentials $\tilde{\omega}_j$, which are the projection of the $\omega_j$ on $W$ along the space of holomorphic forms. Using the cup product matrix, this gives each $\bom_0, \ldots \bom_{g-1}$ as a $\Q$-linear combination of $\tilde{\omega}_g, \ldots, \tilde{\omega}_{2g-1}$. \item Then we compute the projection of $\omega_g, \dots, \omega_{2g-1}$ with respect to the basis $\{\omega_0, \ldots, \omega_{g-1}, \tilde{\omega}_{g},\ldots \tilde{\omega}_{2g-1}\}.$  The linear algebra allows us to write each $\tilde{\omega}_g, \ldots, \tilde{\omega}_{2g-1}$, and consequently, each $\bom_0, \ldots \bom_{g-1}$ as a sum involving a $\Q_p$-linear combination of $\{\omega_0, \ldots, \omega_{g-1}\}$, together with a $\Q$-linear combination of the $\{\omega_g, \ldots, \omega_{2g-1}\}$.\end{enumerate}
\item Compute the double integral for $\tau(z)$ as a power series in $z$ using $\tau(P)$; in particular, the relevant double integrals from $P$ to $z$ are computed as tiny integrals within that disk, using the fact that
\begin{equation}\label{tau}\tau(z) = \tau(P) - 2\sum_{i=0}^{g-1}\left(\int_P^z \omega_i\bom_i+ \frac{1}{2}\int_P^z \omega_i\int_{-P}^P \bom_i\right). \end{equation} The tiny double integrals are computed as in \cite[Algorithm 5.1]{balakrishnan:iterated}.
\end{enumerate}

We now compute the necessary precision for each local height and Coleman integral discussed in this section.  We begin by recalling the statement of Proposition 18 of \cite{BBK09}, very slightly adapted to fit our purposes:
\begin{proposition}[Proposition 18, \cite{BBK09}]\label{prop18} Let $\int_P^Q \omega$ be a
  tiny integral in a residue disk not equal to the disk at infinity, with $P,Q \in \UU(\Z_p)$ accurate to $n$ digits of precision. Let $(x(t),y(t))$ be the local interpolation between $P$ and $Q$ and $\omega = g(x,y)dx$ a differential of the second kind such that $h(t) = g(x(t),y(t))$ belongs to $\Z_p[[t]]$. If we truncate $h(t)$ modulo $t^m$, then the computed value of the integral $\int_P^Q \omega$ will be correct to $\min\{n, m+1-\lfloor \log(m+1)\rfloor\}$ digits of precision.
\end{proposition}

The corresponding statement for tiny double integrals easily follows:

\begin{corollary} Let $\int_P^Q \omega \eta$ be a tiny double integral in a residue disk
  not equal to the disk at infinity, with $P,Q \in \UU(\Z_p)$ accurate to $n$ digits of precision. Let $(x(t),y(t))$ be the local interpolation between $P$ and $Q$ and  $\omega \int \eta= g(x,y)dx$ a differential of the second kind such that $h(t) = g(x(t),y(t))$ belongs to $\Z_p[[t]]$. If we truncate $h(t)$ modulo $t^m$, then the computed value of the integral $\int_P^Q \omega \eta$ will be correct to $\min\{n, m+1-\lfloor \log(m+1)\rfloor\}$ digits of precision.
\end{corollary}

As a special case of \cite[Proposition 19]{BBK09}, we have

\begin{proposition}\label{prop19}  Let $\int_P^Q \omega$ be
  a Coleman integral of $\omega$ a differential of the first kind or a differential of the second kind without poles in the disks of integration. Suppose $P,Q \in
  \UU(\Z_p)$ are in non-Weierstrass residue disks and given accurate to $n$ digits of precision. Let
  $\Frob$ be the matrix of the action of Frobenius on the basis differentials. Set $B =
  \Frob^t - I_{2g \times 2g}$ and let $m _1= v_p(\det(B))$. Then the computed value of the integral $\int_P^Q \omega$ will be accurate to $n - \max\{m_1, \lfloor \log n \rfloor \} $digits of precision.
\end{proposition}

\begin{proof}The Coleman integral is computed via a linear system, where the quantities
  involved are an inverted matrix $B^{-1}$ times a vector of constants involving tiny
  integrals and integrals of exact forms evaluated at points. Suppose the entries of $B$
  are computed to precision $n$. Then taking $B^{-1}$, we must divide by $\det(B)$, which
  lowers the precision by $m_1 = v_p(\det(B))$. By Proposition \ref{prop18}, computing tiny integrals with the series expansions truncated modulo $t^{n-1}$ gives a result precise up to $n - \lfloor \log n \rfloor$ digits. Thus the value of the integral $\int_P^Q \omega$ will be correct to $n - \max\{m_1, \lfloor \log n\rfloor\}$ digits of precision.
\end{proof}

Finally, we discuss precision of local height integrals:

\begin{proposition}Let $P = (a,b)  \in \UU(\Z_p)$ be a point accurate to $n$ digits of
  precision. Let $\tau(P)$ be the local height pairing at $p$ of $(P) - (\infty)$ with itself.
  Let $\Frob$ be the matrix of the action of Frobenius on the basis differentials. Set $B
  = \Frob^t - I_{2g \times 2g}$ and let $m _1= v_p(\det(B))$. There exists an algorithm to compute $\tau(P)$ to $n - \max\{m_1, \lfloor \log n\rfloor\}$ digits of precision.
\end{proposition}

\begin{proof} Let $a_{2g+1}$ denote the leading coefficient of the polynomial $f$ defining $X$. Suppose $P$ is a Weierstrass point. Then by \cite[Lemma 4.1]{BBM14}, we have $\tau(P) = \frac{1}{2}(\log_p(f'(a)) + \log_p(a_{2g+1}))$, and we are done.

Now suppose $P$ is a non-Weierstrass point. We discuss the precision, mostly in terms of local coordinates, necessary to compute $\tau(P)$ to $n - \max\{m_1, \lfloor \log n\rfloor\}$ digits of precision.  Throughout, we recall formulas given in \cite[$\S 4$]{BBM14} and \cite[$\S 6$]{Bes-Bal10}. 

We begin with the following equation for $\tau(P)$ from \cite[$\S 4$]{BBM14}:
$$\tau(P) = \frac{1}{4}(\log_p(4b^2) + 2\log_p(a_{2g+1})) + \frac{1}{4}\left(\underbrace{\int_{w(P)}^P \frac{b\, dx}{y(x-a)}}_{\eqref{1}}- \int_{w(P)}^P \eta\right),$$
where $\eta$ is the holomorphic projection of the differential $\frac{b\, dx}{y(x-a)}$,
computed from the splitting~\eqref{eq:decompi}. 
Computing Coleman integrals as in Proposition \ref{prop19}  gives the integral of $\eta$ correct to $n - \max\{m_1, \lfloor \log n\rfloor\}$ digits of precision.   

The integral $\int_{w(P)}^P \frac{b\, dx}{y(x-a)} $ is somewhat more delicate to analyze, since it is the Coleman integral between two different residue disks of a differential form with poles at both of its endpoints, $P$ and $w(P)$. To compute it, the strategy is to break up the path to produce two slightly simpler integrals: the sum of \begin{itemize}\item a tiny integral of a differential form with precisely one pole at an endpoint, which can be computed by normalizing the resulting integral with respect to a choice of cotangent vector and\item a Coleman integral between two different residue disks with poles within the disks of integration (but away from the endpoints).\end{itemize}  

We begin by breaking up the path from $w(P)$ to $P$ using an arbitrary point $Q$ (distinct from $P$) in the disk of $P$:
\begin{align*}\int_{w(P)}^P  \frac{b\, dx}{y(x-a)}  &= \int_{w(P)}^Q  \frac{b\, dx}{y(x-a)} + \int_{Q}^P  \frac{b\, dx}{y(x-a)}  \\
&= \int_{w(Q)}^P \frac{b\, dx}{y(x-a)} + \int_{Q}^P  \frac{b\, dx}{y(x-a)}  \\
&=  \int_{w(Q)}^Q \frac{b\, dx}{y(x-a)} + \int_{Q}^P  \frac{b\, dx}{y(x-a)} + \int_{Q}^P  \frac{b\, dx}{y(x-a)}, 
\end{align*}which gives 
\begin{equation}\label{1} \int_{w(P)}^P \frac{b\, dx}{y(x-a)}  = -2\underbrace{\int_P^Q \frac{b\, dx}{y(x-a)}}_{\eqref{2}} + \underbrace{\int_{w(Q)}^Q \frac{b\, dx}{y(x-a)}}_{\eqref{3}}.\end{equation}
The integral \eqref{2} is now a tiny integral, but it still has a pole at one of its endpoints, $P$.  Let $z = \frac{x-a}{2b}$ be the normalized parameter at $P$, used to normalize the integral, as in \cite[$\S 4$]{BBM14}. Writing $\frac{b\, dx}{y(x-a)}$ in terms of this parameter yields a Laurent series expansion of the form $(z^{-1} + a_0 + a_1 z  + \cdots )dz$, because $\frac{b\, dx}{y(x-a)}$ has a simple pole with residue 1 at $P$. The normalized integral $\int_P^Q\frac{b\, dx}{y(x-a)}$  is the normalized integral of $\frac{b\, dx}{y(x-a)}$ evaluated at $Q$; in particular, the normalization gives that the constant term of the integral with respect to $z$ is 0, which lets us write it as $$\int_P^Q\frac{b\, dx}{y(x-a)} = \log_p(z) + a_0z + \frac{a_1}{2}z^2 + \cdots \mid_{z = z(Q)}.$$ Changing to the parameter $t = (x-a) = 2bz$, we obtain
\begin{align}\label{2}\int_P^Q\frac{b\, dx}{y(x-a)} &= \log_p\left(\frac{t}{2b}\right) + a_0\left(\frac{t}{2b}\right) + \frac{a_1}{2}\left(\frac{t}{2b}\right)^2 + \cdots \mid_{t = t(Q)} \nonumber \\
&= \log_p(t(Q)) - \log_p(2b) + \int_{0}^{t(Q)} \left(\frac{b\, dx(t)}{y(t)(x(t)-a)} - t^{-1} \right)dt.\end{align} Since the precision of this last integral is found by computing a truncation of the tiny integral modulo $t^{m'-1}$ where $m' -\log(m') \geq n$, the value of \eqref{2} will be correct to $n$ digits of precision.

Finally, the integral \eqref{3} can be computed as discussed in \cite[$\S 4.2$]{BBM14}.  We recall the formula  \cite[$(4.4)$]{BBM14} used to evaluate it:

\begin{equation}\label{3}\int_{w(Q)}^Q \frac{b\, dx}{y(x-a)} =
  \frac{1}{1-p}\left(\Psi(\alpha) \cup \Psi(\beta) + \sum_{A \in \mathcal{S}}
  \Res_A\left(\alpha \int \beta\right) - 2\int_{Q}^{\phi(Q)}
  \frac{b\, dx}{y(x-a)}\right),\end{equation} 
  where $\alpha = \phi^*\left(\frac{b\, dx}{y(x-a)}\right) - p\left(\frac{b\,
  dx}{y(x-a)}\right)$, $\beta$ is a
  differential form with residue divisor $(Q) - (w(Q))$, $\mathcal{S}$ is the set of
  closed points on $X\times\Q_p$, and $\Psi(\xi)$ is the logarithm of a differential $\xi$ as in \cite[Prop 2.3]{Bes-Bal10}, where the map $\Psi$ takes a differential of the third kind to a differential of the second kind modulo exact differentials and is computed using definite Coleman integrals of a basis of $\hdr^1(X)$.

 Note that the differential $\alpha$ is chosen so that it is ``essentially'' of the second
 kind, so that within each residue disk, the residue along a sufficiently narrow annulus at the boundary of the disk is 0; for further details, see \cite[Remark 4.9]{Bes-Bal10}. We have that $\Psi(\alpha)$ and $\Psi(\beta)$ have precision prescribed by definite Coleman integrals of a basis of $\hdr^1(X)$; that is, computing the Coleman integrals as in Proposition \ref{prop19}, this is $n - \max\{m_1, \lfloor \log n\rfloor\}$. 

For $\sum_{A \in \mathcal{S}}\Res(\alpha\int\beta)$, we have three cases: 
\begin{enumerate}[(1)]
  \item $A$ finite Weierstrass, 
  \item $A$ non-Weierstrass and not in residue disks containing the support of
the residue divisor of $\beta$,
\item $A$ non-Weierstrass and in disks containing the support of the residue divisor of $\beta$.
\end{enumerate}
 Case (1) is the most straightforward. To
compute $\Res_A(\alpha \int\beta)$ to $n$ digits of $p$-adic precision, by
\cite[Proposition 6.5]{Bes-Bal10}, we need to compute the local coordinate at $A$,
$(x(t),y(t))$ to $\O(t^{2pn - p - 3})$. The dependence on $p$ in the exponent arises from analyzing the contributing terms in a $t$-adic expansion of $\alpha$, which amounts to a binomial expansion of $\phi^*\left(\frac{b\, dx}{y(x-a)}\right)$ where $\phi(x) = x^p$. For case (2), suppose $A$ is defined over a degree
$d$ extension $K/\Q_p$ and we have a working precision of $n$ digits, so that $A$ has precision
$dn$ in terms of a uniformizer $\pi$ of $K$. As in \cite[Corollary 6.4]{Bes-Bal10}, let
$m''$ be such that $dm'' + 1 - \lfloor \log(dm'' + 1) \rfloor \geq n$. Then computing
$\beta$ in terms of a local coordinate $(x(t),y(t))$ at a point $B$ in the disk of $A$
defined over $\Q_p$, we have that $\beta$ can be truncated modulo $t^{dm''}$ for
$\sum_{U_A} \Res(\alpha\int\beta)$ to be correct to $n$ digits, where $U_A$ is the residue
disk containing $A$. For case (3), we use \cite[Lemma 4.2]{BBM14}  to first rewrite $\int \beta$, and then use the truncation in case (2).

Lastly, we analyze the precision of $\int_{Q}^{\phi(Q)} \frac{b\, dx}{y(x-a)}$. We take
this opportunity to make a small correction to the discussion above \cite[Lemma
4.2]{BBM14}; indeed, the same technique described in \cite[Lemma 4.2]{BBM14} is also used
to compute this integral, since the residue divisor of $\frac{b\, dx}{y(x-a)}$ is $(P) -
(w(P))$, and by construction, $Q, \phi(Q)$, and $P$ are all in the same residue disk.  So by computing the integrand in terms of a local coordinate at $Q$ and truncating modulo $t^{m'}$ as above, we have that the integral will be correct to $n$ digits of precision.
\end{proof}

\subsection{Computing the intersection patterns}\label{sec:T}
We now discuss the computation of a set $T\subset\Q_p$ which contains the values of $\rho$
on $p$-integral points. 
By Lemma~\ref{L:tdec}, it suffices to compute the set $\TT$ of intersection patterns,
since we can take $T=\log_p(\TT)$. 
Hence we compute the sets $T(q)$ of possible values of $D_{P,q}^2$
(see~\eqref{E:dp}), as $P$ runs through $\UU(\Z_q)$, where $q$ is a prime of bad reduction.
In fact, it follows from Lemma~\ref{L:tdec} that it suffices to do this for $q=2$ and for primes $q$
which divide the leading coefficient of $f$ or satisfy $v_q(\disc(f))\ge 2$, since all
other primes are not very bad.

So let $q$ be such a prime; we will omit it from the notation in the following. 
We first compute a strong desingularization $\XX$ of the Zariski closure of $X\times\Q_q$ over $\Spec(\Z_q)$.
Suppose that, as a divisor, $\XX_q$ is given by $\sum^n_{i=1}
a_i\Gamma_i$, where $\Gamma_1,\ldots,\Gamma_n$ are the irreducible components of $\XX_q$
and $a_1,\ldots,a_n$ are positive integers.
For $i,j \in \{1,\ldots,n\}$ let $m_{ij}=(a_i\Gamma_i\,.\,a_j\Gamma_j)$, so that
$M=(m_{ij})_{i,j}$ is the intersection matrix of $\XX_q$.
Let $M^+$ denote the Moore-Penrose pseudoinverse~\cite{penrose:inverse} of $M$.

Recall the adjunction formulas on $\XX_q$: If $\KK\in\Div(\XX)\otimes\Q$ is a canonical
$\Q$-divisor on $\XX$ (meaning $\O(\KK)$ is isomorphic to the relative dualizing sheaf of $\XX$
over $\Spec(\Z_q)$),
then we have
\begin{equation}\label{adj-form}
    (\KK\,.\,\Gamma_i) = -\Gamma_i^2 + 2p_a(\Gamma_i) - 2
\end{equation}
for all $i$, where $p_a(\Gamma_i)$ denotes the arithmetic genus of $\Gamma_i$.

For $P \in \UU(\Z_q)$, we now give simple formulas for $D^2_P$ in terms of linear algebra on the
special fiber $\XX_q$.
Recall from~\cite[Lemma~3.4]{BBM14} and the proof of~\cite[Proposition~3.3]{BBM14} that if
$P \in X(\Q_q)$, then $D^2_P$ is given by
\begin{equation}\label{eq:D_sq-formula}
D^2_P =  \Phi((P)-(\infty))^2 - (P\,.\,\div_\XX(\omega_0)) -
(\infty\,.\,\div_\XX(\omega_{g-1})).
\end{equation}
We consider each term on the right hand side of~\eqref{eq:D_sq-formula} individually,
drawing on and extending~\cite[\S5]{BBM14}.
As these terms only depend on the component that the section $P$ intersects, we can compute
$T(q)$ by computing them for each component of multiplicity~1.

\subsubsection{Computing $\Phi((P)-(\infty))^2$}
Let $u(P)$ denote the column
vector whose $i$th entry is $((P)-(\infty)\,.\,a_i\Gamma_i)$.
Then we have
\begin{equation}\label{phi_formula}
    \Phi((P)-(\infty))^2=u(P)^TM^+u(P)\,.
\end{equation}
by~\cite[(5.3)]{BBM14}.

\subsubsection{Computing $(P\,.\,\div_\XX(\omega_0))$}
Recall from~\cite[\S5.1]{BBM14} that \[\div_\XX(\omega_0) = (2g-2) \infty+ V\,,\] where
$V$ is vertical.
As $P\in\UU(\Z_q)$, the section associated to it does not 
intersect the section $\infty$, so it follows from~\cite[\S5.1]{BBM14} that
\begin{equation}\label{eq:pomv}
  (P\,.\,\div_\XX(\omega_0))= (P\,.\,V) = (P\,.\,V'),
\end{equation}
for such a point $P$. Here $V'$ is the unique vertical $\Q$-divisor such that $(2g-2)
\infty +V'$ is a canonical $\Q$-divisor and such that the support of $V'$ is disjoint from
the component $\Gamma_0$. 
Writing $V' = \sum^n_{i=1}v'_i\Gamma_i$, $v'_i \in \Q$, it therefore suffices to compute
the column vector $v' = (v'_1,\ldots,v'_n)^t$.

Suppose that $\Gamma_{i_0}$ dominates the Zariski closure of $\XX$ and that the section $\infty$
intersects $\Gamma_{i_\infty}$, where $i_0,\, i_\infty \in \{1,\ldots,n\}$.
Then $v'$  can be computed as follows:
Let $b$ denote $(b_1,\ldots,b_n)^t$, where
\[
    b_i = -\Gamma_i^2 + 2p_a(\Gamma_i) - 2 - (2g-2) (\infty\,.\,\Gamma_i).
\]
By~\eqref{adj-form} we have
\[
    v' = M'\cdot b,
\]
where $M'$ is the matrix obtained from $M$ by setting the $i_\infty$-th row and column
equal to~0.

\subsubsection{Computing $(\infty\,.\,\div_\XX(\omega_{g-1}))$}
The divisor of the function $x^{g-1} \in \Qpb(X)$ is equal to $(g-1)D_0-(2g-2)(\infty)$,
where $D_0 = (P_0) + (w(P_0))\in \Div(X\otimes\Q_q)$ for a point $P_0 \in X(\bar{\Q}_q)$ with $x$-coordinate~0.
Thus we have
\[
    \div_\XX(\omega_{g-1}) = (g-1)D_0 + W
\]
for some vertical divisor $W$.
In analogy with~\eqref{eq:pomv}, we have 
\[
    (\infty\,.\,\div_\XX(\omega_{g-1}))= (\infty\,.\,W) = (\infty\,.\,W'),
\]
where 
\[
  W' = \sum^n_{i=1}w'_i\Gamma_i,\;\;w'_1,\ldots,w'_n \in \Q,
\]
is the unique vertical $\Q$-divisor such that $(2g-2)D_0 + W'$ is a canonical
$\Q$-divisor and such that the support of $W'$ is disjoint from the component $\Gamma_0$.
If $c = (c_1,\ldots,c_n)^t$ is defined by
\[
    c_i = -\Gamma_i^2 + 2p_a(\Gamma_i) - 2 - (g-1)D_0\,.\,\Gamma_i),
\]
then it follows from~\eqref{adj-form} that 
\[
     (w'_1,\ldots,w'_n)^t = M'\cdot c.
\]

\subsection{Computing the solutions to $\rho(z) \in T$} \label{sec:solve}
Here we describe how to find all $p$-adic solutions to $\rho(z) \in T$ in $\UU(\Z_p)$ to
precision $p^N$.  
\begin{enumerate}\item We begin by enumerating the list of residue disks not including the
      disk at infinity: that is, the set $X(\F_p) \setminus\infty$. 
\item Compute the dual basis $\{\bom_i\}$.
\item Initialize sets $\setz = \emptyset$ and $\setz = \emptyset$. \item For each disk, we do the following:
\begin{enumerate}\item If the disk is Weierstrass, we take $P$ to be the characteristic 0 Weierstrass point in the disk. If the disk is not Weierstrass, we take $P$ to be an arbitrary lift to $\Z_p$ of the $\F_p$-point.
\item Compute $\tau(P)$, a local coordinate $(x(z),y(z))$ at $P$, and $f_i(z), \tau(z)$ as in $\S\S\ref{sec:fi}-\ref{sec:tau}$.
\item Use the values of $\alpha_{ij}$ as computed in $\S\ref{sec:alphas} $ to produce the power series expansion 
\begin{equation}
    \rho(z) = \tau(P) - 2\sum_{i=0}^{g-1}\left(\int_P^z \omega_i\bom_i+ \frac{1}{2}\int_P^z \omega_i\int_{-P}^P \bom_i\right)- \sum_{0 \le i\le j \le g-1} \alpha_{ij} f_i(z)f_j(z)
\end{equation}
\item For each value of $t \in T$, solve the equation $\rho(z) \equiv  t \pmod{p^N}$. For
  each $z_i$ such that $\rho(z_i) \equiv  t \pmod{p^N}$, check whether $\rho'(z_i) \equiv
  0 \pmod{p^N}$. If $\rho'(z_i) \not\equiv 0 \pmod{p^N}$, use the local coordinate at $P$ to find the corresponding
  $y$-coordinate of the point on $X$ in the disk with $x$-coordinate $x(z_i)$ and append
  $((x(z_i),y(z_i)),1)$ to the set $\setz$. If $\rho'(z_i) \equiv 0\pmod{p^N}$, approximate $m=\ord_{z=z_i}\rho(z)$ using our working
  precision and append $((x(z_i),y(z_i)), m)$ to the set $\setz$. 
  \end{enumerate}
\end{enumerate}


\section{Complementing quadratic Chabauty} \label{sec:cliffhanger}

Using the algorithms described in Section~\ref{sec:computing}, we can compute a finite set of $p$-adic points on $X$ which contains
the set $\UU(\Z)$ of integral points on $X$, provided the conditions of Theorem~\ref{cutting-function} are satisfied.
More precisely, we can compute the set 
$\TT$ of intersection patterns, and we can find all solutions $z \in\UU(\Z_p)$ to the equations
\begin{equation}\label{E:sols}
  \rho(z) = \log_p(d),\;\;d \in \TT
\end{equation}
to $N$ digits of precision.
Theorem~\ref{cutting-function} implies that the set of solutions will contain the set of integral points on $X$.

For every solution $z$ to~\eqref{E:sols}, there are three possible scenarios:
\begin{enumerate}[(a)]
  \item $z$ lifts to a point in $\UU(\Z)$,
  \item $z$ lifts to a point  in $X(\Q)\setminus \UU(\Z)$, or
  \item $z$ does not lift to a point  in $X(\Q)$.
\end{enumerate}
In cases (a) and (b), it is usually trivial to find the corresponding rational point.
The reason is that because of our assumption $g>1$, we do not expect $X(\Q)$ to contain any points of large height.
So quadratic Chabauty leaves us with a list of integral points $\seta\subset \UU(\Z)$, and we want to prove that
this list is complete.

There are essentially two problems we have to solve.
The first one is caused by the fact that we only compute our solutions to our finite
working precision $p^N$.
\begin{problem}\label{prob:nearby}
Given a solution $z_0$ lifting to a point $P_0 \in X(\Q)$, we have to show that there are no integral points $\ne
P_0$ congruent to $z_0$ modulo $p^N$.
\end{problem}
If the power series expansion of $\rho-t$ in the residue disk containing $z_0$ only has a
simple zero at $z_0$, then we are obviously done.

So suppose that 
$\ord_{z=z_0} \rho(z)-t = m >1$.
In practice, our approach is to simply increase precision until $z_0$ separates into $m$ different points.
If this does not work for reasonably small precision $p^{\tilde{N}}$, then we are in one
of the following two cases:
\begin{enumerate}[(i)]
  \item $z_0$ really occurs as a multiple solution of $\rho(z) - t$.
  \item $z_0$ is $p$-adically closer than $p^{\tilde{N}}$ to another point which satisfies $\rho(z) - t = 0$.
\end{enumerate}
In case (ii), we will of course detect this by increasing precision, though we have no \emph{a
priori} bound on when to stop (if $z_0$ is integral, we could derive an enormous bound
using linear forms in logarithms, but this would not be practical).
However, neither (i) nor (ii) occurred in our examples, and we do not expect this to occur in
practice.

Here is a heuristic explanation for this:
If the original expansion of $\rho$ in the disk of $z_0$ did not use $z_0$ as a basepoint, recompute $\rho(z) = t$ at $z_0$.
Call $F(z) = \rho(z) - t$. To state the simplest case of the problem, suppose we are trying to resolve a potential double root; that is, our expansion of $F$ yields $F(z) = a_2z^2 + \cdots$: Either we actually have a double root at $z_0$ or $a_1$ is $p$-adically too small to show up as a coefficient of the power series, given our working precision.
Note that each element of the dual basis $\bom_i$ is a $\Q_p$-linear combination of the
$g$ regular 1-forms (from the projection along the unit root subspace, using Frobenius)
and a $\Q$-linear combination of the $g$ meromorphic 1-forms $\{\omega_{g}, \ldots \omega_{2g-1}\}$ (from the cup product matrix).  So, using a certain choice of cotangent vector $b_0$ where necessary, as in \cite[$\S\S 2-4$]{BBM14}, we have 
\begin{align*}F(z) &:=  \rho(z) -t\\
 &= -2\int_{b_0}^z\sum_{i=0}^{g-1} \omega_i \bom_i
-\sum_{0\le i\leq j\leq g-1}\alpha_{ij}f_i(z)f_j(z)-t\\
&= -2\int_{b_0}^z\sum_{i=0}^{g-1} \omega_i\sum_{j = g}^{2g-1}b_{ij}\omega_j
-2\int_{\infty}^z\sum_{i=0}^{g-1} \omega_i\sum_{j = 0}^{g-1}b_{ij}\omega_j -\sum_{0\le i\leq j\leq g-1} \alpha_{ij}\int_{\infty}^z\omega_i\int_{\infty}^z\omega_j -t\\
&= ( q_1z + q_2 z^2 + \cdots) + ( s_1 z + s_2 z^2 + \cdots),
\end{align*}
where the first power series has rational coefficients $q_i$ (the sums of $\Q$-linear combinations of double integrals which are computed using a local coordinate expansion at the integral point $z_0$)  and the second power series has $p$-adic coefficients $s_i$. This means that $F(z)$ has linear coefficient 0 or very close to 0, that is, either $q_1 = s_1$ (and that $s_1$ is actually a rational number) or $q_1 - s_1 \approx 0$, or that $s_1 \approx q_1$, that is, $s_1$ is very close to a rational number, both of which are scenarios that seem highly unlikely.

So Problem~\ref{prob:nearby} should not lead to any practical difficulties.
However, there is a much more serious obstacle we have to overcome.
\begin{definition}\label{D:f}
  Let $\calF_p$ denote the set of solutions to the equations~\eqref{E:sols} which do
not lift to a $\Q$-rational point to our working precision.
\end{definition}
As we do not expect any $\Q$-rational points of large height, we have to solve the
following problem.
\begin{problem}\label{prob:fake}
  Show that no element of $\calF_p$ lifts to an integral point.
\end{problem}
By the discussion above, it suffices to solve Problem~\ref{prob:fake} in order to 
provably find all integral points on $X$.
So how can we attack Problem~\ref{prob:fake}?
This question will be answered in Section~\ref{sec:global}, after we establish the basics of the Mordell-Weil sieve in the next section.


\section{The Mordell-Weil sieve}\label{sec:mw-sieve}
From now on, we suppose that the genus $g$ of $X$ is at least~2.  We recall the Mordell-Weil sieve, a method to obtain results on rational points of $X(\Q)$ 
using information about $J(\Q)$ and information over finite fields.
It was developed by Scharaschkin in his PhD thesis~\cite{Scharaschkin:Thesis} and 
adapted and applied to great success by a number of authors, for instance Flynn~\cite{Flynn:Hasse}, Poonen-Schaefer-Stoll~\cite{PSS:Twists}
and Bruin-Stoll~\cite{Bruin-Stoll:Experiment, Bruin-Stoll:MWSieve}. 
The Mordell-Weil sieve is particularly useful for proving that $X(\Q)$ is empty, but it
has many other applications as well, see~\cite[\S4.2]{Bruin-Stoll:MWSieve}.
In our discussion we follow~\cite{Bruin-Stoll:MWSieve}.

The general idea is as follows:
Let $M$ be a positive integer and let $$\pi \colon J(\Q) \to J(\Q)/MJ(\Q)$$ be the canonical epimorphism.
Suppose that $C_M \subset J(\Q)/MJ(\Q)$ is a set of residue classes for which we want to
show that no rational point $P \in X(\Q)$ maps to $C_M$ under $\pi\circ\iota$.
In the simplest case we pick a prime $v$ of good reduction and consider the following commutative
diagram:
   \[\xymatrix{
       X(\Q)\ar@{->}[r]^-{\pi\circ\iota} \ar@{->}[d]&
        J(\Q)/MJ(\Q)\ar@{->}[d]^{\alpha_v}\\
        X(\F_v)\ar@{->}[r]_-{\beta_v}& J(\F_v)/MJ(\F_v) } \]
Here $\alpha_v$ is induced by the reduction map $J(\Q)\to J(\F_v)$ and $\beta_v$ is
induced by $\iota \colon X(\F_v) \hookrightarrow J(\F_v)$.
If $\alpha_v(C_M) \cap \beta_v(X(\F_v))=\emptyset$, then we are obviously done.

More generally, we can apply this argument to several primes of good reduction at once:
Let $S$ be a finite set of such primes and consider
the commutative diagram
\[
    \xymatrix{
    X(\Q)\ar@{->}[r]^-{\pi\circ\iota}  \ar@{->}[d]&
    J(\Q)/MJ(\Q)\ar@{->}[d]^{\alpha_S}\\
    \prod_{v\in S}X(\F_v)\ar@{->}[r]_-{\beta_S}&\prod_{v\in S} J(\F_v)/MJ(\F_v)\,.\\
} \]
Then it suffices to show that
\[
    \alpha_S(C_M) \cap \beta_S\left(\prod_{v\in S}X(\F_v)\right)=\emptyset\,.
\]
In other words, we want to find $S$ such that 
\begin{equation}\label{E:ascm}
    A(S,C_M) =  \left\{c \in C_M : \alpha_S(c) \in \beta_S\left(\prod_{v\in
    S}X(\F_v)\right)\right\}
\end{equation}
is empty.
Heuristically, we can estimate the size of $A(S,C_M)$ as follows:
For a prime $v$ of good reduction, we denote by $X_{M,v}=\beta_v(X(\F_v))$ the image of $X(\F_v)$ in
$J(\F_v)/MJ(\F_v)$ and we set
\[
  \gamma(v,M) = \frac{\#X_{M,v}}{\#J(\F_v)/MJ(\F_v)}.
\]
Note that $v$ can only be useful for our purposes if $\gamma(v,M)<1$.
The expected size of $A(S, C_M)$ is
\begin{equation}\label{eq:nsc}
  n(S,C_M) = \#C_M\prod_{v \in S} \gamma(v,M),
\end{equation}
so we want $n(S,C_M)$ to be small.

\begin{remark}\label{R:eff}
In addition to choosing $S$, we can of course also increase $M$ by taking some multiple
$M'$ of $M$ and constructing the set $C_{M'}\subset J(\Q)/M'J(\Q)$ of residue classes which
map to $C_M$ via the Chinese remainder theorem.
It seems very difficult to prove that there are
choices of $S$ and $M'$ so that this approach will show
\[
  \pi(\iota(X(\Q)))\cap C_M = \emptyset
\]
as desired, so that at present the Mordell-Weil sieve is not known to be effective.
See Poonen~\cite{Poonen:Exp} for related heuristics, which predict, in particular, that
such a choice should always exists for curves that have $X(\Q) = \emptyset$.
\end{remark}

We refer to Bruin and Stoll~\cite{Bruin-Stoll:MWSieve} for a thorough discussion of
practical aspects of the Mordell-Weil sieve.
We finally mention that it is also possible to use primes of bad reduction and ``deep''
information for the Mordell-Weil sieve, but we have not tried to use this for our
purposes.
See~\cite{Bruin-Stoll:MWSieve} on a description of these ideas for genus~2.

The practicality of the Mordell-Weil sieve is discussed in~\cite[\S\S7--8]{Bruin-Stoll:MWSieve}.
Essentially, two problems can arise.
First, the set $A(S,C_M)$ can become too large. 
Secondly, the computation of the images of the sets $X(\F_v)$ in $J(\F_v)$ can become infeasible. 
The latter can occur because that computation requires approximately $v$ discrete
logarithms in the group $J(\F_v)$.
We can avoid this issue by only considering primes such that $\#J(\F_v)$ is sufficiently
smooth, meaning that it has no large prime divisors, as in this case we can apply
Pohlig-Hellman reduction~\cite{pohlig-hellman}, so that only a small number
of relatively easy discrete logarithms have to be computed.


\section{Combining quadratic Chabauty and the Mordell-Weil sieve}\label{sec:global}
Now we show how quadratic Chabauty can be combined with the Mordell-Weil
sieve to provably compute all integral points on $X$, where we assume $X$ to have genus at
least~2. See Appendix~\ref{sec:elliptic} for a brief discussion of the genus~1 case.

Recall that in Section~\ref{sec:computing} we discussed algorithms to compute a
finite subset of $\UU(\Z_p)$ to $N$ digits of precision which contains the set $\UU(\Z)$ of integral points, provided the
conditions of Theorem~\ref{cutting-function} are satisfied.
Having identified some of these $p$-adic points as actual rational points, we showed in 
Section~\ref{sec:cliffhanger} that in order to provably compute all integral points, it suffices to 
solve Problem~\ref{prob:fake}:
We need to prove that the set $\mathcal{F}_p$ of remaining $p$-adic points
do not lift to integral points.
In fact, we will discuss how to prove that they do not lift to rational points by 
translating this into a problem that we can solve using the Mordell-Weil sieve.

For simplicity, we first assume that the torsion subgroup of $J(\Q)$ is trivial and that we have
generators $P_1,\ldots,P_g$ of $J(\Q)$ available.
See Remarks~\ref{R:tors} and~\ref{R:index} on how to remove these assumptions.

Suppose that $z \in \UU(\Z_p)$ is accurate to $N$ digits of precision.
If $z$ lifts to rational point $P\in X(\Q)$, then there are integers $a_1,\ldots,a_g$ such that 
\begin{equation}\label{eq:zais}
  \iota(P) = a_1P_1+\ldots+a_gP_g.
\end{equation}
By properties of the Coleman integral, this implies
\[
    f_i(z) = \int^z_{\infty}\omega_i =a_1\int_{P_1}\omega_i+\ldots+
    a_g\int_{P_g}\omega_{i}\;\textrm{ for all } i \in \{0,\ldots,g-1\}.
\]
Let $(\tilde{a}_1,\ldots,\tilde{a}_g) \in \Z/p^N\Z$ be defined by
\begin{equation}\label{eq:sle}
  \left(\begin{array}{c}
    \tilde{a}_1\\
      \vdots\\
      \tilde{a}_g
    \end{array}\right) = 
  \left(\begin{array}{ccc}
      \int_{P_1}\omega_0  & \cdots & \int_{P_g}\omega_0\\
                   \vdots & \ddots & \vdots\\
  \int_{P_1}\omega_{g-1}  & \cdots & \int_{P_g}\omega_{g-1}\\
    \end{array}\right)^{-1}\cdot\left(\begin{array}{c}
        \int^z_{\infty} \omega_0\\
        \vdots\\
        \int^z_{\infty} \omega_{g-1}\\
    \end{array}\right).
\end{equation}
Note that if~\eqref{eq:zais} holds, then we must have $a_i \equiv \tilde{a}_i \pmod{p^N}$ for all
$i\in \{1,\ldots,g\}$.

So, if we are given $z$ as above, and we want to prove that $z$ does not lift to a
rational point, we compute $(\tilde{a}_1,\ldots,\tilde{a}_g) \in \Z/p^N\Z$, noting 
that if $d$ is the valuation of the determinant of the matrix we are inverting, 
we need its entries to precision $p^{N+d}$. 
Then it suffices to show that there is no tuple $(a_1,\ldots,a_g) \in \Z^g$ such that
\begin{enumerate}[(i)]
  \item $a_i \equiv \tilde{a}_i \pmod{p^N}$ for all $i \in\{1,\ldots,g\}$,  and
    \item $a_1P_1+\ldots+a_gP_g$ is in the image of $X(\Q)$ in $J(\Q)$ under the embedding
        $\iota$.
\end{enumerate}
The tuple $(\tilde{a}_1,\ldots,\tilde{a}_g)$ induces a residue class $c \in J(\Q)/p^NJ(\Q)$ via
\[
  c = \tilde{a}_1P_1+\ldots+\tilde{a}_gP_g;
\]  
we say that $c$ {\em corresponds to $z$}.

Then, denoting by $\pi \colon J(\Q) \to J(\Q)/p^NJ(\Q)$ the canonical
epimorphism, we want to show that there is no point $P \in X(\Q)$ whose image under
$\pi\circ\iota$ maps to the residue class $c$ in $J(\Q)/p^NJ(\Q)$ corresponding to
$z$ -- which is precisely the type of problem the Mordell-Weil sieve is likely to
solve! We set $M=p^N$, let $C_M\subset J(\Q)/MJ(\Q)$ denote the set of residue classes corresponding to
elements of $\mathcal{F}_p$ and look for a set $S$ of primes of good reduction such that
$A(S,C_M)=\emptyset$, where $A(S,C_M)$ is defined in~\eqref{E:ascm}.

Unfortunately, as stated, this approach is not likely to succeed in practice. The reason
is that for $M$ a prime power the set $A(S,C_M)$ is unlikely to be empty, 
as we cannot expect $\#J(\F_v)$ to be divisible by $p$ for many small primes $v$.
Of course, we can simply take a multiple $M'$ of $p^N$ as discussed in Remark~\ref{R:eff}.
However, we would like to apply our method to curves of genus (and hence rank) at least~4,
so this strategy would likely lead to large sets $A(S',C'_M)$. Moreover, the discrete
logarithms needed to compute the images of $X(\F_v)$ in $J(\F_v)$ would become infeasible,
since the shape of $M$ would force us to consider primes $v$ such that $\#J(\F_v)$ has large
prime divisors and the computation of discrete logarithms (in fact explicit computations
in $J(\F_v)$ in general) becomes slow as $g$ and $v$ become large.

Fortunately, it turns out that we can actually do much better.
Namely, we use the following strategy:
We combine the information coming from quadratic Chabauty for {\em several} primes of good ordinary
reduction $p_1,\ldots,p_s$ to $N_1,\ldots,N_s$ digits of precision, respectively.
Let $m$ be a small integer coprime to the $p_k$ and set
\begin{equation}\label{E:m}
M=m \cdot p_1^{N_1}\cdots p_s^{N_s}.
\end{equation}
If $P \in \UU(\Z)$, then there is an intersection pattern $d \in \TT$ such that we have 
\[
  \rho_{p_k}(P) = \log_{p_k}(d)\;\textrm{ for all }k \in \{1,\ldots,s\}.
\]
We can partition $\mathcal{F}_{p_k}$ into subsets 
\[
  \mathcal{F}_{p_k,d} = \mathcal{F}_{p_k} \cap \rho_{p_k}^{-1}(\log_{p_k}(d)),
\]
where $d$ runs through $\TT$.
For an intersection pattern $d$, we define $C_d$ to be the
set of residue classes $c \in J(\Q)/MJ(\Q)$ such that
$c$ corresponds to an element of $\mathcal{F}_{p_k,d}$ 
for every $k \in \{1,\ldots,s\}$, without any condition modulo $m$.
Finally, we set
\begin{equation}\label{E:cm}
  C_M := \bigcup_{d \in \TT} C_d
\end{equation}
and apply the Mordell-Weil sieve with these choices of $M$ and $C_M$ for some set $S$ of
primes of good reduction so that $A(S,C_M)$ is likely to be empty.

\begin{remark}\label{R:tors}
If there is nontrivial rational torsion on $J$, then we have to replace $C_M$ by the set
of all classes $c+Q \in J(\Q)/MJ(\Q)$, where $c \in C_M$ and $Q$ runs through $J(\Q)_{\mathrm{tors}}\setminus\{O\}$.
This increases the size of $C_M$, but otherwise does not cause any problems.
\end{remark}

\begin{remark}\label{R:index}
If the points $P_i$ only generate a subgroup of finite, but unknown index, then  it suffices to show that
this index is coprime to all $v \in S$ used in the Mordell-Weil sieve computation \emph{a
posteriori}, because then $\langle P_1,\ldots, P_g\rangle$ and
$J(\Q)/J(\Q)_{\mathrm{tors}}$ have the same image in $J(\F_v)$.
\end{remark}


\section{Choosing the parameters for quadratic Chabauty and the Mordell-Weil sieve}\label{sec:applying_mws}
The previous section described a general approach to combining quadratic Chabauty with the
Mordell-Weil sieve in order to provably compute all integral points on $X$.
Now we discuss how to make this practical.
To this end, we need to establish how to choose the various parameters
for the Mordell-Weil sieve and quadratic Chabauty computations.
More precisely, we need to choose
\begin{itemize}
  \item the primes $p_1,\ldots,p_s$ of good ordinary reduction,
  \item the numbers $N_1,\ldots,N_s$ of digits of precision,
  \item the integer $m$,
  \item the set $S$ of primes of good reduction,
\end{itemize}
so that the set $A(S,C_M)$ has a good chance to be empty, where $M=m\cdot p_1^{N_1}\cdots p_s^{N_s}$
and $C_M$ is defined in~\eqref{E:cm}.
If we can then show that $A(S,C_M)$ is indeed empty, we have proved that our
list of known integral points is complete.
To this end, we will make all the choices listed above at the same time, as opposed to
choosing the $p_k$, the $N_k$ and $m$ \emph{a priori} and then looking for a 
suitable set $S$.

With notation as in Section~\ref{sec:mw-sieve}, the expected size 
of the set $A(S,C_M)$ is
$$n(S,C_M) =  \#C_M\prod_{v \in S} \gamma(v,M).$$
We would like to find $S$ and $M$, so that $n(S,C_M)$ is small, where $C_M$ is as in~\eqref{E:cm}, before we actually start
the quadratic Chabauty computations.

So suppose we have candidate $S$ and $M$.
It is straightforward to compute $\gamma(v,M)$ for given $v$ and $M$,
see~\cite[\S3]{Bruin-Stoll:MWSieve}.
However, the size of $C_M$ is given by
\[
  \#C_M = \#J(\Q)_{\mathrm{tors}}/MJ(\Q)_{\mathrm{tors}}\cdot
  \#J(\Q)/mJ(\Q)\cdot\sum_{d \in
  \TT}\prod^s_{k=1}\#\mathcal{F}_{p_k,d}.
\]
Hence we cannot compute $\#C_M$ without already knowing the sizes of the sets
$\mathcal{F}_{p_k,d}$ for the various $d \in \TT$. 
But these sets come out of the quadratic Chabauty computation for $p_k$!
So instead, we settle for an integer $\theta(C_M)$ such that we expect $$\#C_M \le\theta(C_M)$$
to hold, using the following strategy:

We first compute the very bad primes and the set $\TT$ of intersection patterns.
Then we find all small integral points on $X$, say of logarithmic height at most $B_2$.
For every such point $P$, we use~\eqref{eq:rho_away} to find the element $d \in \TT$ such that 
$\rho_{p_k}(P) = \log_{p_k}(d)$ for all primes $p_k$. Let $\sigma_d$ denote the number of known integral points
corresponding to $d$; this can be computed easily once we have computed the regular models
$\XX$ as in Section~\ref{sec:T}.
Experimental data suggests that we should expect at most two solutions to the equation
\[\rho_{p_k}(z) = \log_{p_k}(d)\] in every affine residue disk.
So we expect 
\[
\#\mathcal{F}_{p_k,d} \le
  \sum_{d \in \TT} (2\cdot\#X(\F_{p_k}) - 2 -
  \sigma_d)
\]
and hence $\#C_M \le\theta(C_M)$, where
\[
  \theta(C_M) =\#J(\Q)_{\mathrm{tors}}/MJ(\Q)_{\mathrm{tors}}\cdot
  \#J(\Q)/mJ(\Q)\cdot\sum_{d \in
  \TT}\prod^s_{k=1}\left(2\cdot\#X(\F_{p_k})-2-\sigma_d\right) 
\]
Thus we expect the size of $A(S,C_M)$ to be at most
\begin{equation}\label{eq:del}
  n'(S,C_M) =  \theta(C_M) \cdot \prod_{v \in S}
  \gamma(v,M).
\end{equation}
So our goal is to find $S$ and $M$ as above such that $n'(S,C_M)$ is small.
Our approach to this is as follows:
\begin{enumerate}[(a)]
  \item Fix positive constants $\varepsilon$ and $B_1$, let $p_1,\ldots,p_u$ denote the primes $\le B_1$
    of good ordinary reduction, and set $S=\emptyset$.
  \item Append primes $v$ of good reduction to $S$ such that $\#J(\F_v)$ is divisible by at least
    one prime $p_k$ and only has prime divisors below some chosen bound, and such that we have $n'(S',C_{M'})<\varepsilon$, where 
    \begin{itemize}
      \item $M' = M_{l-g-1} = m'\cdot p_1^{N_1}\cdots p_u^{N_u}$,
      \item $\prod_{v \in S} J(\F_v) \cong \Z/M_{1}\Z\times\cdots\times\Z/M_l\Z$, and
       \item  $M_j\mid M_{j+1}$ for $j = 1,\ldots,l-1$.
    \end{itemize}
  \item Find $M=m\cdot p_1^{N_1}\cdots p_s^{N_s}\mid M'$ and $S' \subset S$ such that $n'(S,C_M)<\varepsilon$ and such that
    $A(S,C_M)$ can be computed efficiently.
\end{enumerate}
This is essentially analogous to the techniques proposed by Bruin and Stoll
in~\cite{Bruin-Stoll:MWSieve}.
In particular, see~\cite[\S3.1]{Bruin-Stoll:MWSieve} for an argument why $M'=M_{l-g-1-j}$ is
likely to yield small values of $n(S',C_{M'})$ and~\cite[\S3.2]{Bruin-Stoll:MWSieve} for
criteria which guarantee that $A(S,C_M)$ can be computed efficiently.
Their arguments are detailed for the case $C_M = J(\Q)/MJ(\Q)$, but can be adapted easily
to remain valid in our situation.
In particular, we restrict to primes $v$ such that the prime divisors of $\#J(\F_v)$ are not too large 
to keep the discrete logarithms needed in the Mordell-Weil sieve computation feasible. 
In practice we usually choose $B_1\approx 50$, so
that the quadratic Chabauty computations are reasonably fast. 
Note that usually only small powers of the $p_k$ will divide some $\#J(\F_v)$, so we
need only a few digits of $p$-adic precision.
Using the strategy outlined above, we can keep a good balance between the size of the sets $C_M$, and a
suitable flexibility in choosing $S$.


\section{A method for computing the integral points when $r=g>1$}\label{sec:algorithm}
We collect the techniques described in this paper into a complete method for computing the
set of integral points on a hyperelliptic curve with Mordell-Weil rank equal to its genus.

\begin{algo}[Computing the set of all integral points on $X$]  $\qquad$

\noindent \textbf{Input:} A separable polynomial $f\in \Z[x]$ of degree $2g + 1 \ge 5$ \begin{itemize}\item that does not reduce to a square modulo any prime number and \item
such that the Jacobian $J$ of the hyperelliptic curve $X$ defined by $y^2=f(x)$ has rank $J(\Q)$ equal to $g$.\end{itemize} 
\noindent We assume that an explicit set of independent non-torsion points $P_1, \ldots, P_g \in J(\Q)$ is readily available.\\
\noindent \textbf{Output:} This procedure either computes the set $\UU(\Z)$ of all
integral points on $X$ or terminates with an error.\\
\begin{enumerate}
  \item Choose positive integers $B_1, B_2$ and a positive real number $\varepsilon$.
  \item Compute the rational points on $X$ of height $\le B_2$ and let $\seta$ denote the
    set of integral points among these.
  \item Compute the very bad primes and the set $\TT$ of intersection
    patterns as in $\S\ref{sec:T}$.
  \item Compute the torsion subgroup $J(\Q)_{\mathrm{tors}}$.
  \item Using the strategy discussed in $\S\ref{sec:applying_mws}$, find a set $S$ of
    primes of good reduction and a positive integer $M=m\cdot p_1^{N_1}\cdots p_s^{N_s}$
    such that
    \begin{enumerate}
      \item every $p_k$ is a good ordinary prime $\le B_1$;
      \item $\gcd(p_1\cdots p_s,m)=1$; 
      \item  we have 
    \[
      n'(S,C_M) < \varepsilon,
    \]
    where $C_M$ is as in~\eqref{E:cm}.
    \end{enumerate}
  \item For every $p \in \{p_1,\ldots,p_s\}$, do the following: 
        \begin{enumerate}
          \item Using the estimates of $\S\ref{sec:computing}$, find the working
            precision $p^{N'}$ needed for the results of 7aiii) to be correct to $N$
            digits.
          \item Compute the constants $\alpha_{ij}$ as in $\S\ref{sec:alphas}$. If we
            find that the functions $f_0,\ldots,f_{g-1} \colon J(\Q)\otimes\Q\to\Q_p$ are dependent, use
            Chabauty's method  to compute $X(\Q)$ and exit.
          \item Compute a power series expansion of $f_i(z)$ and $\tau(z)$ in
            every residue disk as in $\S\S\ref{sec:fi}-\ref{sec:tau}$.
        \end{enumerate}
  \item For every $d \in \TT$ do the following:
    \begin{enumerate}
      \item For every $p \in \{p_1,\ldots,p_s\}$, do the following: 
        \begin{enumerate}
      \item Compute the set $\setz$ of solutions (with multiplicities)  to $\rho(z) =
        \log_{p} (d)$ as in $\S\ref{sec:solve}$.
      \item If there is a solution $z_0\in\setz$ with multiplicity $>1$ which lifts to
        a rational point, increase the working precision $p^{N'}$.
        If $N'$ exceeds a pre-determined bound, exit with an
        error;  else go to Step~(6b). 
      \item For every solution $z$ lifting to a point $P\in \UU(\Z)$, set $\seta
        = \seta\,\cup\, \{P\}$.
      \item Compute the residue classes $c \in J(\Q)/p^NJ(\Q)$ corresponding to the points
        $z \in \calF_{p_k,d}$.
    \end{enumerate}
\item Apply the Chinese remainder theorem to find the set $C_d$
  containing all residue classes $c\in J(\Q)/MJ(\Q)$ such that $c$ corresponds to an
  element of $\calF_{p_k,d}$ modulo $p_k^{N_K}$ for all $k$.
\item For every $c' \in J(\Q)_{\mathrm{tors}}/MJ(\Q)_{\mathrm{tors}}\setminus\{O\}$ and
  every $c \in C_d$, append 
  $c + c'\in J(\Q)/MJ(\Q)$ to $C_d$.
    \end{enumerate}
  \item Set $C_M = \bigcup_{d \in \TT} C_d$.
  \item Compute $A(S,C_M)$ using the Mordell-Weil sieve.
  \item  If $A(S,C_M)= \emptyset$, then: 
        \begin{enumerate}
      \item Check if the index
    $(J(\Q)/J(\Q)_{\mathrm{tors}}:\langle P_1,\ldots,P_g\rangle)$ is prime to $v$ for
    every $v\in S$ which was used in the Mordell-Weil sieve.
  \item If yes, return $\seta$. 
  \item Otherwise, saturate $\langle P_1,\ldots, P_g\rangle$ at $v$ for every relevant $v
    \in S$ and go to Step (6) with the new points $P_1,\ldots,P_g$ resulting from the saturation
    process.
    \end{enumerate}
Otherwise, divide $\varepsilon$ by~5. If $\varepsilon$ is below a
    pre-determined bound, exit with an error; else, go to Step (5). \end{enumerate}
 \end{algo}

Our implementation of Step (3) crucially relies on {\tt
Magma}'s {\tt RegularModels} package, which computes a strong desingularization of the Zariski
closure of $X$.
Our implementation of the choice of $S$ and $M$ (Step (5)) and of the actual Mordell-Weil sieve
computation (Step (9)) is based on {\tt Magma} code of Michael Stoll.
See~\cite[\S3.3]{Bruin-Stoll:MWSieve} for a description of the ideas underlying that
implementation.
The $p$-adic analytic steps in the algorithm were implemented in {\tt Sage}.

Regarding Step (10), it often happens that some primes $v \in S$ are superfluous in
the Mordell-Weil sieve computation, so we do not have to check whether they divide the index.
See~\cite[\S13.7]{PSS:Twists} or~\cite{Flynn-Smart} where it is explained how to show
that a certain prime does not divide the index. 
In practice, we usually expect the index to be~1, so we expect Step (10c) to never occur.
See~\cite{Flynn-Smart} for a method that can be used to saturate at $v$ when the genus
is~2.
Alternatively, if we find that some prime $v$ that was used in the Mordell-Weil sieve
computation divides the index, we can try to rerun Step (9) without $v$. If this fails to
show $A(S\setminus\{v\},C_M)=\emptyset$, we can go back to Step (5), keeping the same
constants $B_i$ and $\varepsilon$, but \emph{a priori} excluding all elements known to divide the
index from $S$.


\section{Examples}\label{sec:examples}
Here we give some examples illustrating the techniques described in this paper.
We start by finishing the computation of the integral points on a genus~2 curve, to which
we applied quadratic Chabauty in~\cite{BBM14}.
Then we compute the integral points on a modular curve of genus~3.
Finally, we do the same for a curve of genus~4, which does not seem to be amenable to any
techniques previously available.

\begin{example}
Let $X$ be the genus 2 curve
\[y^2 = x^3(x-1)^2 + 1.\]
In~\cite[\S7.2]{BBM14} we showed the following:
\begin{enumerate}[(i)]
  \item The Jacobian $J$ of $X$ has rank~2 over $\Q$.
  \item $J(\Q)_{\mathrm{tors}} =\{O\}$.
  \item The points $P_1 = [(P) -(\infty)]$ and $P_2 = [(Q) - (R)]$ generate $J(\Q)$, where
    $P = (2,-3), \, Q =(1, -1),\, R =(0,1)$.
  \item The only very bad prime is~2.
\item We have $\TT = \{0, \frac{1}{2}, \frac{2}{3}\}$.
\end{enumerate}
We then computed the solutions to $\rho(z) \in T$ at $p=11$.
Among these solutions, the only ones that appeared to be actual integral points were $
P,\, Q,\, R$ and their images under $w$.

Let us now prove that these are indeed the only integral points on $X$.
Note that this has been proved already by Michael Stoll using the methods
of~\cite{BMSST08}.
We apply quadratic Chabauty for the primes $p_1 = 5$ and $p_2 = 11$, to respective
precision $N_1= 4$ and $N_2 =6$
and run the Mordell-Weil sieve with $M =5^4\cdot 11^6$ and $S = \{17, 863, 7193\}$.
After taking out residue classes containing integral points, we are left with
{209} residue classes in $J(\Q)/MJ(\Q)$; the Mordell-Weil sieve computation then shows that {none} of the 209 residue classes
contain the image of a rational point on the curve.
This proves that our list of integral points is complete.
\end{example}

\begin{example}\label{g3}
Let $X$ be the genus 3 hyperelliptic curve $$y^2 = (x^3+x+1)(x^4+2x^3-3x^2+4x+4).$$ 
This is the curve $C_{496}^J$, a \emph{new modular curve} of level~496
with $\Q$-simple Jacobian in the language  of~\cite{BGGP:Finiteness}.
Applying~2-descent as described in~\cite{stoll:2-descent} and implemented in~{\tt Magma}, we find that the Jacobian $J$ of $X$ has rank~3 over $\Q$. 
The torsion subgroup has order~2, the nontrivial element coming from the
factorisation $(x^3+x+1)(x^4+2x^3-3x^2+4x+4)$. 
Let $P = (-1, 2),\, Q = (0, 2),\, R = (-2, 12),\, S = (3, 62).$ 
We want to prove that, up to the hyperelliptic involution, these are the only integral points on $X$.

The points on the Jacobian represented by differences of known rational points on $X$
generate a subgroup of $J(\Q)/J(\Q)_{\mathrm{tors}}$ of finite index; an explicit set of
generators of this subgroup is given by $$\{P_1 = [(P)-(\infty)],\, P_2 =
[(S)-(w(Q))],\, P_3 = [(w(S)) - (R)]\}.$$
As we do not need it for our method, we did not check whether these points generate the Mordell-Weil group modulo
torsion; however, this should be possible using the techniques described in~\cite{Stoll:KummerG3}.

We apply the method from Section~\ref{sec:algorithm}; 
leading to the following choices:
\begin{enumerate}[(i)]
  \item $M=3\cdot7^3\cdot 17^3\cdot 37^2$, 
  \item $S = \{  5, 41, 607, 617, 1861, 11131, 17209 \}$.
\end{enumerate}
The very bad primes are $2$ and $31$ and the set of intersection patterns is
$$\TT = \left\{(a,b)\, : \,a \in \left\{0, 1,\frac{5}{4}, \frac{7}{4}\right\}, b \in
\left\{0,\frac{1}{2}\right\}\right\}.$$

Now we give some detail on the quadratic Chabauty computation for $p_1=7$. 
We find the following $\Z_7$-points having $\rho$-values in the set $T = \log_7(\TT)$.

\begin{center}
 \begin{tabular}{||c |  r  | c ||}
    \hline
  disk & $x(z)$ & $\rho(z)$ \\
  \hline
$\overline{(3,\pm 1)}$ & $3 + 3 \cdot 7 + O(7^{3})$   &  $0$  \\
& $3 + 3 \cdot 7 + 3 \cdot 7^{2} + O(7^{3})$ & $\frac{1}{2}\log_7(31)$ \\
& $3 + 2 \cdot 7 + 5 \cdot 7^{2} + O(7^{3}) $  & $\log_7(2)$  \\
& $3 + 2 \cdot 7 + 7^{2} + O(7^{3})$  &  $\log_7(2) + \frac{1}{2}\log_7(31)$\\
& $3 + 4 \cdot 7^{2} + O(7^{3}) $ &  $\frac{5}{4}\log_7(2)$\\
& $\mathbf{3 + O(7^{3})} $  &  $\frac{5}{4}\log_7(2) + \frac{1}{2}\log_7(31)$ \\
& $3 + 3 \cdot 7 + 5 \cdot 7^{2} + O(7^{3})  $  &  $\frac{7}{4}\log_7(2)$ \\
& $3 + 3 \cdot 7 + 7^{2} + O(7^{3})$ & $\frac{7}{4}\log_7(2) + \frac{1}{2}\log_7(31)$ \\
\hline
$\overline{(4,\pm 1)}$ &  &   \\
\hline
$\overline{(0,\pm 2)}$ & $4 \cdot 7 + 5 \cdot 7^{2} +  O(7^{3})$ & $0$   \\
& $5 \cdot 7 + 6 \cdot 7^{2} + O(7^{3})$ & $0$\\
& $4 \cdot 7 + 2 \cdot 7^{2} + O(7^{3})$ &
$\frac{1}{2}\log_7(31)$\\
& $5 \cdot 7 + 2 \cdot 7^{2} + O(7^{3}) $ & $ \frac{1}{2}\log_7(31)$\\
& $\mathbf{O(7^3)} $ & $\log_7(2)$\\
& $2 \cdot 7 + 2 \cdot 7^{2}  + O(7^{3}) $ &$\log_7(2)$ \\
& $2 \cdot 7^{2} + O(7^{3})  $ &$\log_7(2)+ \frac{1}{2}\log_7(31)$ \\
& $2 \cdot 7 + O(7^{3})$ & $\log_7(2)+\frac{1}{2}\log_7(31)$\\
& $4 \cdot 7 + O(7^{3})$ & $\frac{7}{4}\log_7(2)$ \\
& $5 \cdot 7 + 4 \cdot 7^{2} + O(7^{3})$ &$\frac{7}{4}\log_7(2)$\\
& $4 \cdot 7 + 4 \cdot 7^{2} + O(7^{3})$ &$\frac{7}{4}\log_7(2) + \frac{1}{2}\log_7(31)$\\
& $5 \cdot 7 + O(7^{3})$ &$\frac{7}{4}\log_7(2) + \frac{1}{2}\log_7(31)$\\
\hline
$\overline{(5,\pm 2)}$ & $5 + 6 \cdot 7 + 7^{2} + O(7^{3}) $ & $0$\\
& $5 + 4 \cdot 7 + 5 \cdot 7^{2} + O(7^{3}) $ & $0$\\
& $5 + 6 \cdot 7 + 4 \cdot 7^{2} + O(7^{3})$ & $\frac{1}{2}\log_7(31)$\\
& $5 + 4 \cdot 7 + 2 \cdot 7^{2} + O(7^{3})$ &  $\frac{1}{2}\log_7(31)$\\
& $5 + 2 \cdot 7 + 2 \cdot 7^{2} + O(7^{3}) $ &  $\frac{5}{4}\log_7(2)$ \\
& $5 + 7 + 2 \cdot 7^{2}  + O(7^{3})$ &  $\frac{5}{4}\log_7(2)$\\
& $5 + 2 \cdot 7 + 7^{2} + O(7^{3}) $ & $\frac{5}{4}\log_7(2) +\frac{1}{2}\log_7(31)$ \\
& $5 + 7 + 3 \cdot 7^{2} + O(7^{3}) $ &  $\frac{5}{4}\log_7(2) + \frac{1}{2}\log_7(31)$\\
& $\mathbf{5 + 6 \cdot 7 + 6 \cdot 7^{2} + O(7^{3})}$ & $\frac{7}{4}\log_7(2)$ \\
& $5 + 4 \cdot 7 + 7^{3} + O(7^{3})$ &  $\frac{7}{4}\log_7(2)$\\
& $5 + 6 \cdot 7 + 2 \cdot 7^{2} + O(7^{3})$ & $\frac{7}{4}\log_7(2)+\frac{1}{2}\log_7(31)$\\
& $5 + 4 \cdot 7 + 4 \cdot 7^{2} + O(7^{3}) $ & $\frac{7}{4}\log_7(2)+\frac{1}{2}\log_7(31)$\\
\hline
   \end{tabular}
\end{center}
\begin{center}
 \begin{tabular}{||c |  r  | c ||}
    \hline
  disk & $x(z)$ & $\rho(z)$ \\

 \hline
$\overline{(6,\pm 2)}$ & $6 + 3 \cdot 7 + O(7^{3}) $  &  $0$ \\
& $6 + 7 + 3 \cdot 7^{2} + O(7^{3})$ & $0$\\
& $6 + 3 \cdot 7 + 3 \cdot 7^{2} + O(7^{3}) $ & $\frac{1}{2}\log_7(31)$ \\
& $6 + 7 +  O(7^{3})$ & $\frac{1}{2}\log_7(31)$\\
& $\mathbf{6 + 6 \cdot 7 + 6 \cdot 7^{2} +O(7^3)}$ & $\frac{5}{4}\log_7(2)$ \\
& $6 + 5 \cdot 7 + 5 \cdot 7^{2} + O(7^{3}) $ & $\frac{5}{4}\log_7(2)$\\
& $6 + 6 \cdot 7 + 5 \cdot 7^{2} + O(7^{3})$ & $\frac{5}{4}\log_7(2)+\frac{1}{2}\log_7(31)$\\
& $6 + 5 \cdot 7 + 6 \cdot 7^{2} + O(7^{3}) $ &  $\frac{5}{4}\log_7(2)+\frac{1}{2}\log_7(31)$\\
& $6 + 3 \cdot 7 + 5 \cdot 7^{2} + O(7^{3})$ & $\frac{7}{4}\log_7(2)$ \\
& $6 + 7 + 5 \cdot 7^{2} + O(7^{3})$ & $\frac{7}{4}\log_7(2)$\\
& $6 + 3 \cdot 7 + 7^{2} +O(7^{3}) $ & $\frac{7}{4}\log_7(2)+\frac{1}{2}\log_7(31)$\\
& $6 + 7 + 2 \cdot 7^{2} +O(7^{3}) $ & $\frac{7}{4}\log_7(2)+\frac{1}{2}\log_7(31)$\\
  \hline
$\overline{(2,\pm 3)}$ & $2 + 7^{2} + O(7^{3})$  & $0$  \\
& $2 + 5 \cdot 7 + 2 \cdot 7^{2} + O(7^{3}) $ & $0$\\
& $2 + O(7^{3})$ &  $\frac{1}{2}\log_7(31)$\\
& $2 + 5 \cdot 7 + 3 \cdot 7^{2} +O(7^{3}) $ &   $\frac{1}{2}\log_7(31)$\\
& $2 + 7 + 4 \cdot 7^{2} +O(7^{3})$ & $\log_7(2)$\\
& $2 + 4 \cdot 7 + 3 \cdot 7^{2} + O(7^{3})$ & $\log_7(2)$\\
& $2 + 7  + O(7^{3})$ &
  $\log_7(2)+\frac{1}{2}\log_7(31)$ \\
& $2 + 4 \cdot 7 +  O(7^{3})$ &
  $\log_7(2)+\frac{1}{2}\log_7(31)$\\
& $2 + 4 \cdot 7^{2} + O(7^{3}) $ &
  $\frac{7}{4}\log_7(2)$\\
& $2 + 5 \cdot 7 + 6 \cdot 7^{2} +O(7^{3})$ & $\frac{7}{4}\log_7(2)$\\
& $2 + 3 \cdot 7^{2} + O(7^{3})$ & $\frac{7}{4}\log_7(2)+\frac{1}{2}\log_7(31)$\\
& $2 + 5 \cdot 7 +  O(7^{3}) $ &
  $\frac{7}{4}\log_7(2)+\frac{1}{2}\log_7(31)$\\
\hline
\end{tabular}
\end{center}
\bigskip

After carrying out the quadratic Chabauty computations for $7,17$ and $37$, we are left with~31488 residue classes modulo $7^3\cdot 17^3\cdot 37^2$ for which we have
to show that they do not contain the image of a rational point.
We also add all residue classes modulo~3, leading to~850176 residue
classes modulo $M = 3\cdot7^3\cdot 17^3\cdot 37^2$. 
Applying the Mordell-Weil sieve, we find that none of these residue classes can contain
the image of a rational point.

Since the index of the subgroup  generated by differences of the known
rational points on $X$ in $J(\Q)/J(\Q)_{\mathrm{tors}}$ is easily shown not to be
divisible by any prime in the set $S$, this proves that
our list of integral points is indeed complete.
\end{example}

\begin{example}\label{g4}
Let $X$ be the hyperelliptic curve $$y^2 = x^4(x-2)^2(x-1)(x+1)(x+2)+4$$ of genus~4. Let 
$P = (0,2),\, Q = (1,2),\, R = (2,-2),\,S = (-1,-2),\, U = (-2,2)$.
We want to show that these points, together with their images under $w$, form the complete
list of integral points.

Using~{\tt Magma}, we show that the rank of the Jacobian over $\Q$ is~4 and that the
torsion subgroup is trivial.
A set of generators of the finite-index subgroup $G$ of $J(\Q)$ generated by differences of the known rational points on $X$
is given by $$\{P_1 = [(P)-(Q)], P_2 = [(R)-(S)], P_3 = [(U)-(w(P))], P_4 = [(w(Q)) -
(w(S))] \}.$$

Because we cannot compute generators for the full Mordell-Weil group in genus~4, the techniques of~\cite{BMSST08} are not applicable.
Moreover, the rank~is equal to the genus, so the method of Chabauty-Coleman is not
applicable.
The $S$-unit techniques mentioned in Section~\ref{sec:previous} fail because in order to
apply them in this case, we would need to compute class groups and fundamental units of number fields of
degree~144. A similar complication arises when trying a combination of covering techniques with
elliptic curve Chabauty.
Hence none of the previously known methods are practical for this example.

First we compute that~2 is the only very bad prime and that
$ \TT = \left\{0,\frac{1}{2}, \frac{12}{7}\right\}$.
We then apply quadratic Chabauty for the good ordinary primes $5,7,11, 13$ and $17$,
to~6,~4,~4,~4 and~4
digits of precision, respectively.
For instance, we find the following $\Z_5$-points having $\rho$-values in the set $T =
\log_5(\TT)$:
\begin{center}
 \begin{tabular}{||c |  r  | c ||}
    \hline
  disk & $x(z)$ & $\rho(z)$ \\
  \hline
  $\overline{(0,\pm 2)}$ & $\mathbf{O(5^6)}$ & $\frac{12}{7}\log_5(2)$ \\
  &$3 \cdot 5 + 5^{2} + 2 \cdot 5^{4} + O(5^{6})$ & $\frac{12}{7}\log_5(2)$\\
  \hline
  $\overline{(1,\pm 2)}$ & $\mathbf{1+ O(5^6)}$& $\frac{1}{2}\log_5(2)$ \\
  &  $1 + 5 + 2 \cdot 5^{2} + 5^{4} + 4 \cdot 5^{5} + O(5^{6})$ & $\frac{1}{2}\log_5(2)$ \\
  & $1 + 4 \cdot 5 + 5^{2} + 2 \cdot 5^{3} + 2 \cdot 5^{5} + O(5^{6})$ &
  $\frac{12}{7}\log_5(2)$\\
  & $1 + 2 \cdot 5 + 2 \cdot 5^{3} + 3 \cdot 5^{4} + 2 \cdot 5^{5} + O(5^{6})$ &
  $\frac{12}{7}\log_5(2)$ \\
  \hline
 $\overline{(2,\pm 2)}$ & $\mathbf{2 + O(5^6)}$ & $\frac{12}{7}\log_5(2)$\ \\
 & $2 + 5 + 2 \cdot 5^{3} + 4 \cdot 5^{5} + O(5^{6})$ &  $\frac{12}{7}\log_5(2)$\\
 \hline
 $\overline{(3,\pm 2)}$   & $3 + 5 + 4 \cdot 5^{2} + 2 \cdot 5^{3} + 5^{4} + O(5^{6})$ & $0$  \\
      & $3 + 5^{2} + 3 \cdot 5^{3} + 2 \cdot 5^{4} + 3 \cdot 5^{5} + O(5^{6})$ &
 $\frac{1}{2}\log_5(2)$ \\
      & $\mathbf{3 + 4 \cdot 5 + 4 \cdot 5^{2} + 4 \cdot 5^{3} + 4 \cdot 5^{4} + 4 \cdot
      5^{5} + O(5^{6})}$ & $\frac{12}{7}\log_5(2)$ \\
      \hline
  $\overline{(4,\pm 2)}$      &  $4 + 5^{2} + 5^{3} + 5^{4} + 3 \cdot 5^{5} + O(5^{6})$ & $0$ \\
          & $4 + 5 + 3 \cdot 5^{4} + O(5^{6})$ & $0$ \\
          & $\mathbf{4 + 4 \cdot 5 + 4 \cdot 5^{2} + 4 \cdot 5^{3} + 4 \cdot 5^{4} + 4
          \cdot 5^{5} + O(5^{6})}$ &  $\frac{1}{2}\log_5(2)$\\
          & $4 + 2 \cdot 5 + 4 \cdot 5^{2} + 5^{3} + 2 \cdot 5^{4} + 5^{5} + O(5^{6})$ &
      $\frac{1}{2}\log_5(2)$\\
\hline
\end{tabular}
\end{center}

\bigskip

This leads to~9660096 residue classes modulo $M = (5\cdot 7\cdot 11\cdot 13\cdot
17)^3$.
Using the set $S=\{13, 19, 83, 103, 167, 727, 971, 2909 \}$, all of whose elements are easily shown
to be coprime to the index $(J(\Q) : G)$, the Mordell-Weil sieve proves that none of these residue classes
contain the image of a rational point on $X$.
\end{example}

\appendix
\section{Computing all integral points on elliptic curves using quadratic Chabauty}\label{sec:elliptic}
Let $X$ be an elliptic curve of rank~1 over $\Q$.
For simplicity, we also assume that $X(\Q)$ is torsion-free.
As described in Section~\ref{sec:previous}, there exist quite efficient algorithms for the
computation of the integral points on $X$; the most practical ones are based on elliptic
logarithms.
Nevertheless, we can use quadratic Chabauty, 
for this purpose as well, as sketched below.

Suppose we have already found the set $\seta\subset \UU(\Z)$ of suspected
integral points on $X,$ and
our goal is to prove the equality $\seta= \UU(\Z)$. 

Let $P$ be the generator of $X(\Q)$. 
The first helpful observation is that we 
only have to consider residue disks not corresponding to affine multiples of the reduction $\tilde{P}$ of
$P$ in the quadratic Chabauty computation.

We define the function
$\log_P$ on $X(\Q_p)$ by
\[
  \log_P(z) = \log_{P,p}(z) = \frac{\int_{\infty}^z \omega_0}{\int_{\infty}^P \omega_0},
\]
noting that this is indeed the discrete logarithm with respect to $P$
for points in $X(\Q)$.

Let $\TT$ be the set of intersection patterns. 
We try to prove that $\seta= \UU(\Z)$ by showing, for each $d\in \TT$,
that there are no points in $\UU(\Z)-\seta$ with intersection pattern $d$.

We use quadratic Chabauty for several primes. Let us first discuss the case of~2 primes.
Following Silverman-Stange~\cite{Silverman-Stange}, we call a pair of distinct primes $(p_1,p_2)$ {\em
amicable}  if 
$$p_2=\#{X}(\F_{p_1})\;\;\textrm{and}\;\;{p_1}=\#{X}(\F_{p_2}).$$
More generally, we call a pair of distinct primes $({p_1},p_2)$ {\em
pseudo-amicable} if
$$p_2\mid\#{X}(\F_{p_1})\;\;\textrm{and}\;\;{p_1}\mid\#{X}(\F_{p_2}).$$

Suppose that ${p_1}$ and $p_2$ are primes of good ordinary reduction such that
$({p_1},p_2)$ is pseudo-amicable.
For simplicity, assume that $p_1 = \#X(\F_{p_2})$.
We first apply quadratic Chabauty to $X$ for ${p_1}$ and $p_2$, keeping in mind
that only residue classes corresponding to affine multiples of $\tilde{P}$ have
to be considered.
Suppose that $d \in \TT$ and that $z\in X(\Q_{p_1})$ is a solution to $\rho_{p_1}(z) =
\log_{p_1} (d)$ which lifts to a rational point on $X$.
Then we have $z = aP$ for some $a \in\Z$. 
The discrete logarithm $a_1$ of $z \bmod {p_1}$ with respect to $\tilde{P}$ is then congruent
to $a$ modulo~${p_2}$.
From $a = \log_{P,p_2}(z) \bmod p_1$ we know the
residue disk $\DD_{z,p_2}$ of the reduction of $z$ modulo~$p_2$ (more generally, if we had
$p_1\mid \#X(\F_{p_2})$, then we would know that the reduction of $z$ modulo~$p_2$ lies in
one of $\#X(\F_{p_2})/p_1$ explicitly given residue disks).

We run through all solutions $w\in X(\Q_{p_2})$ to $\rho_{p_2}(w) = \log_{p_2}(d)$ in the residue disk $\DD_{z,p_2}$ and compute
\[a_2=\log_{P,p_2}(z) \bmod p_2.\]
Finally, we check whether $a_1=a_2$; if not, then we know that $w$ doesn't
correspond to an integral point. If this happens for all such $w$, then we can discard $z$.

\begin{example}\label{ex:elliptic}
  Let $X$ denote the elliptic curve
\[
  y^2 = x^3-4.
\]
Then $X$ has rank~1 and trivial torsion over $\Q$ and $P =(2,2)$ is a generator of
$X(\Q)$.
Using an implementation of the algorithm based on elliptic logarithms mentioned in
Section~\eqref{sec:previous} (for instance in {\tt Magma} or {\tt Sage}), we easily find that the only $\Z$-integral points on $X$
are $\pm P$ and $\pm Q$, where $Q = (5,11)$.
We want to recover this result using the techniques discussed above.
First we compute that~2 is the only very bad prime and that $\TT = \{0,1\}$
and $\#X(\F_v)$ for all primes $v<100$ of good and ordinary reduction.

We find that the pair $(13,7)$ is pseudo-amicable, since $\#\tilde{X}(\F_{13}) = 21$ and
$\#\tilde{X}(\F_{7}) = 13$.
Using the algorithms in Section~\ref{sec:computing} we compute all solutions to
$\rho_{13}(z) \in \{0,\log_{13}(2)\}$.
Since the order of the reduction $\tilde{P}$ of $P$ modulo~13 is~7, it suffices to consider 
the~6 residue disks corresponding to nontrivial multiples of $\tilde{P}$.
We then proceed as follows:
For every solution $z\notin \{\pm P, \pm Q\}$ to $\rho_{13}(z) = \log_{13}(d)$, where $d
\in \TT$, we assume that $z$ does correspond to a
rational point $a\cdot P\in X(\Q)$ and we  compute 
\begin{itemize}
  \item the discrete logarithm  $a'\in \F_7$ of $z \bmod 13$ with respect to $\tilde{P}$;
  \item the residue disk $\DD_{z,7}$ of $z$ mod~7, by computing 
$\int^z_{\infty}\omega\cdot\left(\int^P_{\infty}\omega\right)^{-1}$.
\end{itemize}
So we only have to check for all solutions $w \in \DD_{z,7}$ to $\rho_7(w) = \log_7(t)$,
whether 
\[a''=\int^w_{\infty}\omega\cdot\left(\int^P_{\infty}\omega\right)^{-1} \bmod 7\]
coincides with $a'$; if not, then we know that $z\notin X(\Q)$.

Using this approach, we succeeded in showing that indeed $\pm P, \pm Q$ are the only integral points on $X$.
\end{example}

More generally, we recall~\cite[Sec. 2, Definition]{Silverman-Stange} that an \emph{aliquot
cycle} of length $l$ for $X$ is a finite sequence of primes $\cycp=
(p_1,p_2,\ldots, p_l) $ with the property that $\#X(\F_{p_i}) = p_{i+1}
$, where $i$ is taken modulo $l$. We can similarly define
\emph{pseudo-aliquot cycles} by replacing the condition with the weaker one
asserting only that $p_{i+1}$ divides $\#X(\F_{p_i})$.

\begin{definition}
  Let $\cycp$ be a  pseudo-aliquot cycle of length $l$ for $X$. A \emph{lift}
  of $\cycp$ is a
  collection of points $z_i \in X(\Q_{p_i})$, for each $i=1,\ldots,l$,
  such that $\rho_{p_i}(z_i) = \log_{p_i}(d)$, and 
   such that the entire collection satisfies the 
   following conditions with $n_i = \log_{P,p_i}(z_i)\in \Q_{p_i}$:
  \begin{itemize}
  \item We have $n_i \in \Z_{p_i}$.
  \item  As $p_{i+1}$ divides $\#X(\F_{p_{i}})$, $n_{i+1} P$ is restricted to certain
  residue classes in $X(\Q_{p_{i}})$ and we assume that $z_{i}$
  sits in one such residue class.
  \end{itemize}
\end{definition}

Clearly an integral point $Q\in \UU(\Z)$ with intersection pattern $d\in\TT$ gives
rise to the lift of $\cycp$ where $z_i$ is the image of $Q$ in
$X(\Q_{p_i})$. We will say that $\cycp$ {\em disqualifies} $d$ if there
are no lifts of $\cycp$ other than the ones corresponding to points in
$\seta$. This then proves that there are no points in $\UU(\Z)-A$ with
intersection pattern $d$. If we manage to disqualify all
intersection patterns in $\TT$ we prove that $\seta=\UU(\Z)$.

We have not tried to carry this out for cycles of length $>2$ in practice.

\bibliographystyle{amsplain}
\bibliography{total}

\providecommand{\bysame}{\leavevmode\hbox to3em{\hrulefill}\thinspace}
\providecommand{\MR}{\relax\ifhmode\unskip\space\fi MR }
\providecommand{\MRhref}[2]{%
  \href{http://www.ams.org/mathscinet-getitem?mr=#1}{#2}
}
\providecommand{\href}[2]{#2}
\begin{thebibliography}{10}

\bibitem{BGGP:Finiteness}
Matthew~H. Baker, Enrique Gonz{\'a}lez-Jim{\'e}nez, Josep Gonz{\'a}lez, and
  Bjorn Poonen, \emph{Finiteness results for modular curves of genus at least
  2}, Amer. J. Math. \textbf{127} (2005), no.~6, 1325--1387. \MR{2183527
  (2006i:11065)}

\bibitem{balakrishnan:iterated}
Jennifer~S. Balakrishnan, \emph{Iterated {C}oleman integration for
  hyperelliptic curves}, ANTS-X: Proceedings of the Tenth Algorithmic Number
  Theory Symposium (E.\thinspace{}W. Howe and K.\thinspace{}S. Kedlaya, eds.),
  Open {B}ook {S}eries, vol.~1, Mathematical {S}ciences {P}ublishers, 2013,
  pp.~41--61.

\bibitem{Bes-Bal10}
Jennifer~S. Balakrishnan and Amnon Besser, \emph{Computing local $p$-adic
  height pairings on hyperelliptic curves}, IMRN \textbf{2012} (2012), no.~11,
  2405--2444.

\bibitem{Bes12}
\bysame, \emph{Coleman--{G}ross height pairings and the {$p$}-adic sigma
  function}, J. Reine Angew. Math. \textbf{698} (2015), 89--104.

\bibitem{BBM14}
Jennifer~S. Balakrishnan, Amnon Besser, and J.~Steffen M{\"u}ller,
  \emph{{Q}uadratic {C}habauty: {$p$}-adic heights and integral points on
  hyperelliptic curves}, J. Reine Angew. Math (2015), to appear.

\bibitem{BBK09}
Jennifer~S. Balakrishnan, Robert~W. Bradshaw, and Kiran~S. Kedlaya,
  \emph{Explicit {C}oleman integration for hyperelliptic curves}, Algorithmic
  number theory, Lecture Notes in Comput. Sci., vol. 6197, Springer, Berlin,
  2010, pp.~16--31. \MR{2721410 (2012b:14048)}

\bibitem{bdckw}
Jennifer~S. Balakrishnan, Ishai Dan-Cohen, Minhyong Kim, and Stefan Wewers,
  \emph{A non-abelian conjecture of {B}irch and {S}winnerton-{D}yer type for
  hyperbolic curves}, Preprint (2014), 1--38, \url{arxiv:1209.0640}.

\bibitem{BKK11}
Jennifer~S. Balakrishnan, Kiran~S. Kedlaya, and Minhyong Kim, \emph{Appendix
  and erratum to ``{M}assey products for elliptic curves of rank 1''}, J. Amer.
  Math. Soc. \textbf{24} (2011), no.~1, 281--291. \MR{2726605}

\bibitem{BaMuSt12}
Jennifer~S. Balakrishnan, J.~Steffen M\"{u}ller, and William Stein, \emph{A
  {$p$}-adic analogue of the conjecture of {B}irch and {S}winnerton-{D}yer for
  modular abelian varieties}, Math. Comp. (2015), to appear.

\bibitem{Bes00}
Amnon Besser, \emph{{$p$}-adic {A}rakelov theory}, J. Number Theory
  \textbf{111} (2005), no.~2, 318--371. \MR{MR2130113}

\bibitem{magma}
W.~Bosma, J.~Cannon, and C.~Playoust, \emph{The {M}agma algebra system. {I}.
  {T}he user language}, J. Symbolic Comput. \textbf{24} (1997), no.~3--4,
  235--265, Computational algebra and number theory (London, 1993). \MR{1 484
  478}

\bibitem{Bruin:EllipticChabauty}
Nils Bruin, \emph{Chabauty methods using elliptic curves}, J. Reine Angew.
  Math. \textbf{562} (2003), 27--49. \MR{2011330 (2004j:11051)}

\bibitem{Bruin-Elkies:trinomials}
Nils Bruin and Noam~D. Elkies, \emph{Trinomials {$ax^7+bx+c$} and {$ax^8+bx+c$}
  with {G}alois groups of order 168 and {$8\cdot168$}}, Algorithmic number
  theory ({S}ydney, 2002), Lecture Notes in Comput. Sci., vol. 2369, Springer,
  Berlin, 2002, pp.~172--188. \MR{2041082 (2005d:11094)}

\bibitem{Bruin-Stoll:Experiment}
Nils Bruin and Michael Stoll, \emph{Deciding existence of rational points on
  curves: an experiment}, Experiment. Math. \textbf{17} (2008), no.~2,
  181--189. \MR{2433884 (2009d:11100)}

\bibitem{Bruin-Stoll:Descent}
\bysame, \emph{Two-cover descent on hyperelliptic curves}, Math. Comp.
  \textbf{78} (2009), no.~268, 2347--2370. \MR{2521292 (2010e:11059)}

\bibitem{Bruin-Stoll:MWSieve}
\bysame, \emph{The {M}ordell-{W}eil sieve: proving non-existence of rational
  points on curves}, LMS J. Comput. Math. \textbf{13} (2010), 272--306.
  \MR{2685127 (2011j:11118)}

\bibitem{BMSST08}
Y.~Bugeaud, M.~Mignotte, S.~Siksek, M.~Stoll, and S.~Tengely, \emph{Integral
  points on hyperelliptic curves}, Algebra Number Theory \textbf{2} (2008),
  no.~8, 859--885. \MR{2457355 (2010b:11066)}

\bibitem{Chab41}
C.~Chabauty, \emph{Sur les points rationnels des courbes alg\'ebriques de genre
  sup\'erieur \`a l'unit\'e}, C. R. Acad. Sci. Paris \textbf{212} (1941),
  882--885. \MR{0004484 (3,14d)}

\bibitem{Col85a}
R.~Coleman, \emph{Effective {C}habauty}, Duke Math. J. \textbf{52} (1985),
  no.~3, 765--770. \MR{808103 (87f:11043)}

\bibitem{Col-Gro89}
R.~Coleman and B.~Gross, \emph{$p$-adic heights on curves}, Algebraic number
  theory (J.~Coates, R.~Greenberg, B.~Mazur, and I.~Satake, eds.), Advanced
  Studies in Pure Mathematics, vol.~17, Academic Press, Boston, MA, 1989,
  pp.~73--81. \MR{92d:11057}

\bibitem{David95}
Sinnou David, \emph{Minorations de formes lin\'eaires de logarithmes
  elliptiques}, M\'em. Soc. Math. France (N.S.) (1995), no.~62, iv+143.
  \MR{1385175 (98f:11078)}

\bibitem{Flynn:Coverings}
E.~Victor Flynn, \emph{Coverings of curves of genus 2}, Algorithmic number
  theory ({L}eiden, 2000), Lecture Notes in Comput. Sci., vol. 1838, Springer,
  Berlin, 2000, pp.~65--84. \MR{1850599 (2002f:11074)}

\bibitem{Flynn:Hasse}
\bysame, \emph{The {H}asse principle and the {B}rauer-{M}anin obstruction for
  curves}, Manuscripta Math. \textbf{115} (2004), no.~4, 437--466. \MR{2103661
  (2005j:11047)}

\bibitem{Flynn-Smart}
E.~Victor Flynn and Nigel~P. Smart, \emph{Canonical heights on the {J}acobians
  of curves of genus {$2$} and the infinite descent}, Acta Arith. \textbf{79}
  (1997), no.~4, 333--352. \MR{1450916 (98f:11066)}

\bibitem{Flynn-Wetherell:Challenge}
E.~Victor Flynn and Joseph~L. Wetherell, \emph{Covering collections and a
  challenge problem of {S}erre}, Acta Arith. \textbf{98} (2001), no.~2,
  197--205. \MR{1831612 (2002b:11088)}

\bibitem{Gross86}
B.~H. Gross, \emph{Local heights on curves}, Arithmetic geometry ({S}torrs,
  {C}onn., 1984), Springer, New York, 1986, pp.~327--339. \MR{MR861983}

\bibitem{HKK}
N.~Hirata-Kohno and T.~Kovacs, \emph{Computing {$S$}-integral points on
  elliptic curves of rank at least 3}, RIMS Kokyuroku \textbf{1898} (2014),
  92--102.

\bibitem{Kim05}
Minhyong Kim, \emph{The motivic fundamental group of
  {$\mathbb{P}^1-\{0,1,\infty\}$} and the theorem of {S}iegel}, Invent. Math.
  \textbf{161} (2005), no.~3, 629--656. \MR{2181717 (2006k:11119)}

\bibitem{kim10}
\bysame, \emph{Massey products for elliptic curves of rank 1}, J. Amer. Math.
  Soc. \textbf{23} (2010), no.~3, 725--747. \MR{2629986}

\bibitem{Lan88}
S.~Lang, \emph{Introduction to {A}rakelov theory}, Springer-Verlag, New York,
  1988. \MR{89m:11059}

\bibitem{PMC}
William McCallum and Bjorn Poonen, \emph{The method of {C}habauty and
  {C}oleman}, Explicit methods in number theory, Panor. Synth\`eses, vol.~36,
  Soc. Math. France, Paris, 2012, pp.~99--117. \MR{3098132}

\bibitem{Mueller-Stoll}
J.\thinspace{}Steffen M\"uller and Michael Stoll, \emph{Canonical heights on
  genus two {J}acobians}, In preparation, 2015.

\bibitem{penrose:inverse}
R.~Penrose, \emph{A generalized inverse for matrices}, Proc. Cambridge Philos.
  Soc. \textbf{51} (1955), 406--413. \MR{0069793 (16,1082a)}

\bibitem{PZGH99}
Attila Peth{\H{o}}, Horst~G. Zimmer, Josef Gebel, and Emanuel Herrmann,
  \emph{Computing all {$S$}-integral points on elliptic curves}, Math. Proc.
  Cambridge Philos. Soc. \textbf{127} (1999), no.~3, 383--402. \MR{1713117
  (2000f:11069)}

\bibitem{pohlig-hellman}
Stephen~C. Pohlig and Martin~E. Hellman, \emph{An improved algorithm for
  computing logarithms over {GF}$(p)$ and its cryprographic significance}, IEEE
  Trans. Information Theory \textbf{24} (1978), 106--110.

\bibitem{Poonen:Exp}
Bjorn Poonen, \emph{Heuristics for the {B}rauer-{M}anin obstruction for
  curves}, Experiment. Math. \textbf{15} (2006), no.~4, 415--420. \MR{2293593
  (2008d:11062)}

\bibitem{PSS:Twists}
Bjorn Poonen, Edward~F. Schaefer, and Michael Stoll, \emph{Twists of {$X(7)$}
  and primitive solutions to {$x^2+y^3=z^7$}}, Duke Math. J. \textbf{137}
  (2007), no.~1, 103--158. \MR{2309145 (2008i:11085)}

\bibitem{Scharaschkin:Thesis}
Victor Scharaschkin, \emph{Local-global problems and the {B}rauer-{M}anin
  obstruction}, ProQuest LLC, Ann Arbor, MI, 1999, Thesis (Ph.D.)--University
  of Michigan. \MR{2700328}

\bibitem{Silverman-Stange}
Joseph~H. Silverman and Katherine~E. Stange, \emph{Amicable pairs and aliquot
  cycles for elliptic curves}, Exp. Math. \textbf{20} (2011), no.~3, 329--357.
  \MR{2836257 (2012g:11109)}

\bibitem{Smart:S-integral}
Nigel~P. Smart, \emph{{$S$}-integral points on elliptic curves}, Math. Proc.
  Cambridge Philos. Soc. \textbf{116} (1994), no.~3, 391--399. \MR{1291748
  (95g:11050)}

\bibitem{Smart98}
\bysame, \emph{The algorithmic resolution of {D}iophantine equations}, London
  Mathematical Society Student Texts, vol.~41, Cambridge University Press,
  Cambridge, 1998. \MR{1689189 (2000c:11208)}

\bibitem{sage}
W.\thinspace{}A. Stein et~al., \emph{{S}age {M}athematics {S}oftware ({V}ersion
  6.5)}, The Sage Development Team, 2015, \url{http://www.sagemath.org}.

\bibitem{Stoll:H1}
Michael Stoll, \emph{On the height constant for curves of genus two}, Acta
  Arith. \textbf{90} (1999), no.~2, 183--201. \MR{1709054 (2000h:11069)}

\bibitem{stoll:2-descent}
\bysame, \emph{Implementing 2-descent for {J}acobians of hyperelliptic curves},
  Acta Arith. \textbf{98} (2001), 245--277.

\bibitem{Stoll:H2}
\bysame, \emph{On the height constant for curves of genus two. {II}}, Acta
  Arith. \textbf{104} (2002), 165--182.

\bibitem{Stoll:KummerG3}
\bysame, \emph{An explicit theory of heights for hyperelliptic jacobians of
  genus three}, Preprint (2014),
  \url{http://www.mathe2.uni-bayreuth.de/stoll/papers/Kummer-g3-hyp-2014-05-15.pdf}.

\bibitem{Stroeker-Tzanakis}
R.~J. Stroeker and N.~Tzanakis, \emph{Solving elliptic {D}iophantine equations
  by estimating linear forms in elliptic logarithms}, Acta Arith. \textbf{67}
  (1994), no.~2, 177--196. \MR{1291875 (95m:11056)}

\end{thebibliography}

\end{document}